\documentclass[11pt,reqno]{amsart}
 \usepackage{amssymb,amsfonts,amscd,amsbsy}
\usepackage[mathscr]{eucal}
\usepackage{url}

\usepackage{amsmath}
\usepackage{pst-knot}
\usepackage{pstricks}
\usepackage{pst-text}
\usepackage{pstricks-add}
\usepackage{pst-plot}
\usepackage{pst-poly}

\newtheorem{thm}{Theorem}[section]

\newtheorem{lem}[thm]{Lemma}

\theoremstyle{definition}

\numberwithin{equation}{section}
\makeatletter
\def\imod#1{\allowbreak\mkern5mu({\operator@font mod}\,\,#1)}
\makeatother

\begin{document}

\title[Rogers-Ramanujan type identities for alternating knots]{Rogers-Ramanujan type identities for alternating knots}

\author{Adam Keilthy and Robert Osburn}

\address{25 Turnberry, Baldoyle, Dublin 13, Ireland}

\address{School of Mathematical Sciences, University College Dublin, Belfield, Dublin 4, Ireland}

\email{adam.keilthy@gmail.com}

\email{robert.osburn@ucd.ie}

\subjclass[2010]{Primary: 33D15; Secondary: 05A30, 57M25}
\keywords{$q$-series identities, $q$-series transformations, Bailey pairs, alternating knots}

\date{\today}

\dedicatory{Dedicated to  Wen-Ching Winnie Li on the occasion of her birthday}

\begin{abstract}
We highlight the role of $q$-series techniques in proving identities arising from knot theory. In particular, we prove Rogers-Ramanujan type identities for alternating knots as conjectured by Garoufalidis, L{\^e} and Zagier.
\end{abstract}

\maketitle

\section{Introduction}

Two of the most important results in the theory of $q$-series are the classical Rogers-Ramanujan identities which state that 

\begin{equation} \label{rr}
\sum_{n \geq 0} \frac{q^{n^2 + sn}}{(q)_{n}} = \frac{1}{(q^{1+s}; q^5)_{\infty} (q^{4-s}; q^{5})_{\infty}}
\end{equation}

\noindent where $s=0$ or $1$ and

\begin{equation*}
(a)_n = (a;q)_n = \prod_{k=1}^{n} (1-aq^{k-1}),
\end{equation*}

\noindent valid for $n \in \mathbb{N} \cup \{\infty\}$. In 1974, Andrews \cite{andrews74} obtained a generalization of (\ref{rr}) to odd moduli, namely for all $k \geq 2$, $1 \leq i \leq k$, 

\begin{equation} \label{ag}
\sum_{n_1, n_2, \dotsc, n_{k-1} \geq 0} \frac{q^{N_1^2 + N_2^2 + \cdots + N_{k-1}^2 + N_i + N_{i+1} + \cdots + N_{k-1}}}{(q)_{n_1} (q)_{n_2} \cdots (q)_{n_{k-1}}} = \frac{(q^i; q^{2k+1})_{\infty} (q^{2k+1-i}; q^{2k+1})_{\infty} (q^{2k+1}; q^{2k+1})_{\infty}}{(q)_{\infty}}
\end{equation}

\noindent where $N_{j}=n_j + n_{j+1} + \cdots + n_{k-1}$. There has been recent interest in the appearance of these (and similar) identities in knot theory. For example, Hikami \cite{hik} considered (\ref{rr}) from the perspective of the colored Jones polynomial of torus knots while Armond and Dasbach \cite{ad} gave a skein-theoretic proof of (\ref{ag}). For similar identities related to false theta series, see \cite{mh} and for other connections between $q$-series and quantum invariants of knots, see \cite{bhl}--\cite{gk}, \cite{gt}, \cite{hik1} and \cite{lz}.

In this paper, we consider recent work in \cite{gl} whereby the $q$-multisums $\Phi_{K}(q)$ and $\Phi_{-K}(q)$ were associated to a given alternating knot $K$ and its mirror $-K$. The $q$-multisum $\Phi_{K}(q)$ occurs as the $0$-limit (or ``tail'') of the colored Jones polynomial of $K$ (see Theorem 1.10 in \cite{gl}). In Appendix D of \cite{gl}, Garoufalidis and L{\^e} (with Zagier) conjectured evaluations of $\Phi_{K}(q)$ for 21 knots and of $\Phi_{-K}(q)$ for 22 knots in terms of modular forms and false theta series and state ``every such guess is a $q$-series identity whose proof is unknown to us". Before stating these conjectures, we recall some notation from \cite{gl}. For a positive integer $b$, we define

\begin{equation*}
h_{b}= h_{b}(q) = \sum_{n \in \mathbb{Z}} \epsilon_{b}(n) q^{\frac{bn(n+1)}{2} - n}
\end{equation*}

\noindent where

\begin{equation*}
 \epsilon_{b}(n) = \left\{
  \begin{array}{ll}
    (-1)^n & \text{if $b$ is odd,}\\
    1 & \text{if $b$ is even and $n \geq 0$,} \\
    -1 & \text{if $b$ is even and $n < 0.$} 
  \end{array} \right. 
\end{equation*}

\noindent Note that $h_1(q)=0$, $h_2(q)=1$ and $h_3(q)=(q)_{\infty}$. For an integers $p$, $a$ and $b$, let $K_p$ denote the $p$th twist knot obtained by $-1/p$ surgery on the Whitehead link and $T(a,b)$ the left-handed $(a,b)$ torus knot. The 43 conjectures from \cite{gl} are as follows:

\begin{table}[h]
 \begin{tabular}{|c|c|c|}
    \hline
    $K$ & $\Phi_{K}(q)$ & $\Phi_{-K}(q)$ \\ \hline
    $3_1$ & $h_3$ & 1 \\ 
    $4_1$ & $h_3$ & $h_3$ \\
    $5_1$ & $h_5$ & 1 \\
    $5_2$ & $h_4$ & $h_3$ \\
    $6_1$ & $h_5$ & $h_3$ \\
    $6_2$ & $h_3 h_4$ & $h_3$ \\
    $6_3$ & $h_{3}^2$ & $h_{3}^2$ \\
    $7_1$ & $h_7$ & 1 \\
    $7_2$ & $h_6$ & $h_3$ \\
    $7_3$ & $h_5$ & $h_4$ \\
    $7_4$ & $h_4^2$ & $h_3$ \\
    $7_5$ & $h_3 h_4$ & $h_4$ \\
    $7_6$ & $h_3 h_4$ & $h_3^2$ \\
    $7_7$ & $h_3^3$ & $h_3^2$ \\
    $8_1$ & $h_7$ & $h_3$ \\
    $8_2$ & $h_3 h_6$ & $h_3$ \\
    $8_3$ & $h_5$ & $h_5$ \\
    $8_4$ & $h_3$ & $h_4 h_5$ \\
    $8_5$ & $?$ & $h_3$ \\
    $K_p$, $p>0$ & $h_{2p}$ & $h_3$ \\
    $K_p$, $p<0$ & $h_{2|p|+1}$ & $h_3$ \\
    $T(2,p)$, $p>0$ & $h_{2p+1}$ & 1 \\
     
    \hline
  \end{tabular}
\caption{\label{tbl}}
\end{table}

Here, we have corrected the entries for $6_1$, $7_3$, $8_1$, $8_4$, $8_5$, $K_{p}$, $p<0$ (and their mirrors) and $7_5$ in Appendix D of \cite{gl}. Three of these Rogers-Ramanujan type identities, namely 

\begin{equation} \label{three}
\Phi_{3_1}(q)=h_3, \quad \Phi_{4_1}(q)=h_3 \quad \text{and} \quad \Phi_{6_3}(q)=h_3^2
\end{equation}

\noindent have been proven by Andrews \cite{and}. Motivated by his work (and in conjunction with (\ref{three})), we prove the following result.

\begin{thm} \label{main}
The identities in Table \ref{tbl} are true.
\end{thm}

In principle, one can use either Theorem 5.1 of \cite{ad} or Theorem 4.12 of \cite{mh} to give a skein-theoretic proof of Theorem \ref{main}. Here, we have chosen to highlight the role of $q$-series techniques in proving such identities. For example, one can use the Bailey machinery to quickly prove identity (2.7) in \cite{mh}. The paper is organized as follows. In Section 2, we provide the necessary background on $q$-series identities and the Bailey machinery. In Section 3, we clarify the construction of the $q$-multisums $\Phi_{K}(q)$ and $\Phi_{-K}(q)$ from \cite{gl} (see also \cite{gt}). In Section 4, we prove Theorem \ref{main}. It is interesting to note that the proofs for $5_1$ and $-8_4$ require (\ref{rr}) while those for $7_1$ and $T(2,p)$ utilize (\ref{ag}). Although, one can simplify $\Phi_{8_5}(q)$ using the techniques in this paper, a conjectural evaluation is still currently unknown. Moreover, it is not known for a general alternating knot $K$ if $\Phi_{K}(q)$ reduces as in the current pleasant situation.

\section{Preliminaries}

We first recall five $q$-series identities. The first two are due to Euler (see II.1 and II.2, page 236 in \cite{gr}), the third is the $z=1$ case of Lemma 2 in \cite{and}, the fourth is the $q$-binomial theorem (see II.4, page 236 in \cite{gr}) and the fifth is the Jacobi triple product (see II.28, page 239 in \cite{gr}):

\begin{equation} \label{e1}
\sum_{n=0}^{\infty} \frac{t^n}{(q)_{n}} = \frac{1}{(t)_{\infty}},
\end{equation}

\begin{equation} \label{e2}
\sum_{n=0}^{\infty} \frac{(-1)^n t^n q^{n(n-1)/2}}{(q)_{n}} = (t)_{\infty},
\end{equation}

\begin{equation} \label{andy} 
\sum_{n=0}^{\infty} \frac{q^{n^2 + An}}{(q)_{n} (q)_{n+A}} = \frac{1}{(q)_{\infty}}
\end{equation}

\noindent for any integer $A$,

\begin{equation} \label{qbt}
\sum_{n=0}^{\infty} \frac{(-1)^n t^n q^{\frac{n(n-1)}{2}}}{(q)_{n} (q)_{K-n}} = \frac{(t)_{K}}{(q)_{K}}
\end{equation}

\noindent and

\begin{equation} \label{jtp}
\sum_{n \in \mathbb{Z}} z^n q^{n^2} = (-zq; q^2)_{\infty} (-q/z; q^2)_{\infty} (q^2; q^2)_{\infty}.
\end{equation}

\noindent Here and throughout, we use the convention that

\begin{equation*} 
\frac{1}{(q)_{n}} = 0
\end{equation*}

\noindent for $n<0$. In addition, one can easily check that for $a$, $b \geq 0$,

\begin{equation} \label{negab}
\frac{(q^{-a-b})_{a}}{(q)_{a}} = (-1)^{a} q^{-\frac{a(a+1)}{2} - ab} \frac{(q)_{a+b}}{(q)_{a}(q)_{b}}.
\end{equation}

We now derive a key result which follows from a generalization of Sears' transformation (see III.15, page 242 in \cite{gr}).

\begin{lem} \label{key} For any $n>2$ and integers $c_k$,
\begin{equation*}
\sum_{a \geq 0} (-1)^{na}  \frac{q^{\frac{na(a+1)}{2} -a + a  \sum\limits_{k=1}^{n-1} c_{k}}}{(q)_{a}\prod\limits_{k=1}^{n-1}(q)_{a+c_{k}}} = \frac{1}{(q)_{\infty}} \sum_{i_{1}, \dotsc, i_{n-2} \geq 0} (-1)^{\sum\limits_{k=1}^{n-2}\sum\limits_{j=1}^{k}i_{j}} \frac{q^{\frac{1}{2}\sum\limits_{k=1}^{n-2}\bigl(\sum\limits_{j=1}^{k}i_{j}\bigr)\bigl(1+\sum\limits_{j=1}^{k}i_{j} \bigr)+ \sum\limits_{k=2}^{n-1}\sum\limits_{j=1}^{k-1}c_{k}i_{j}}}{\prod\limits_{k=1}^{n-2}(q)_{i_{k}}\prod\limits_{k=1}^{n-2} (q)_{c_{k} +\sum\limits_{j=1}^{k}i_{j} }}.
\end{equation*}
\end{lem}

\begin{proof} We first use that 

\begin{equation*}
\lim_{t \to 0} \, \Bigl(\frac{1}{t}\Bigr)_{n}t^n = (-1)^{n} q^{\frac{n(n-1)}{2}},
\end{equation*}

\noindent then apply Corollary 1 in \cite{ab} and simplify to obtain

\begin{equation*}
\begin{aligned}
& \sum_{a \geq 0} (-1)^{na}  \frac{q^{\frac{na(a+1)}{2} -a + a  \sum\limits_{k=1}^{n-1} c_{k}}}{(q)_{a}\prod\limits_{k=1}^{n-1}(q)_{a+c_{k}}}  =  \frac{1}{\prod\limits_{k=1}^{n-1}(q)_{c_{k}}} \lim_{t \to 0} \sum_{a \geq 0} \frac{(\frac{1}{t})_{a}^{n} \,  t^{na} \,  q^{a\bigl(n-1 + \sum\limits_{k=1}^{n-1}c_{k}\bigr)}}{(q)_{a}\prod\limits_{k=1}^{n-1}(q^{c_{k}+1})_{a}} \\
& = \frac{1}{\prod\limits_{k=1}^{n-1}(q)_{c_{k}}} \lim_{t \to 0} \frac{(tq^{c_{n-1}+1})_{\infty}(t^{n-1}q^{n-1 + \sum\limits_{k=1}^{n-1}c_{k}})_{\infty}}
{(q^{c_{n-1}+1})_{\infty}(t^{n}q^{n-1 +\sum\limits_{k=1}^{n-1}c_{k}})_{\infty}} \\
& \times \sum_{i_{1}, \dotsc, i_{n-2} \geq 0} \frac{(tq^{c_{2} +1})^{i_1}(tq^{c_{3} +1})^{i_1+i_2}\cdots (tq^{c_{n-1} +1})^{i_1+i_2+\dotsc +i_{n-2}}}{(q)_{i_1}(q)_{i_2}\cdots (q)_{i_{n-2}}} \\
& \times \frac{(\frac{1}{t})_{i_1}(\frac{1}{t})_{i_1 + i_2}\cdots (\frac{1}{t})_{i_1 + i_2 + \dotsc + i_{n-2}}}{(q^{c_1 + 1})_{i_1}(q^{c_2 + 1})_{i_1+i_2}\cdots(q^{c_{n-2} + 1})_{i_1+i_2+\ldots +i_{n-2}}}\\
& \times \frac{(tq^{c_1 + 1})_{i_1}\cdots (tq^{c_{n-2} +1 })_{i_{n-2}}(tq^{c_1 + 1})_{i_1}(t^{2}q^{2+c_1+c_2+i_1})_{i_2}\cdots (t^{n-2}q^{n-2 + c_1 + \ldots + c_{n-2} + i_1 + \ldots + i_{n-3}})_{i_{n-2}}}{(t^{n-1}q^{n-1 + c_1 + \ldots + c_{n-1}})_{i_1 + \ldots + i_{n-2}}} \\
& = \frac{1}{(q)_{\infty}} \sum_{i_{1}, \dotsc, i_{n-2} \geq 0} (-1)^{\sum\limits_{k=1}^{n-2}\sum\limits_{j=1}^{k}i_{j}} \frac{q^{\frac{1}{2}\sum\limits_{k=1}^{n-2}\bigl(\sum\limits_{j=1}^{k}i_{j}\bigr)\bigl(1+\sum\limits_{j=1}^{k}i_{j} \bigr)+ \sum\limits_{k=2}^{n-1}\sum\limits_{j=1}^{k-1}c_{k}i_{j}}}{\prod\limits_{k=1}^{n-2}(q)_{i_{k}}\prod\limits_{k=1}^{n-2} (q)_{c_{k} +\sum\limits_{j=1}^{k}i_{j} }}.
\end{aligned}
\end{equation*}

\end{proof}

We now recall the Bailey machinery as initiated by Bailey and Slater in the 1940's and 50's and perfected by Andrews in the 1980's (for further details, see \cite{An2}, \cite{An3} or \cite{war}). A pair of sequences $(\alpha_n,\beta_n)_{n \geq 0}$ satisfying

\begin{equation} \label{pairdef}
\beta_n = \sum_{k=0}^n \frac{\alpha_k}{(q)_{n-k}(aq)_{n+k}}
\end{equation} 

\noindent is called a {\it Bailey pair relative to $a$}. If $(\alpha_n,\beta_n)_{n \geq 0}$ is a Bailey pair relative to $a$, then so is  $(\alpha_n^{'},\beta_n^{'})_{n \geq 0}$ where

\begin{equation} \label{alphaprimedef}
\alpha_n^{'} = \frac{(b)_n(c)_n(aq/bc)^n}{(aq/b)_n(aq/c)_n}\alpha_n
\end{equation} 

\noindent and

\begin{equation} \label{betaprimedef}
\beta_n^{'} = \sum_{k=0}^n\frac{(b)_k(c)_k(aq/bc)_{n-k} (aq/bc)^k}{(aq/b)_n(aq/c)_n(q)_{n-k}} \beta_k.
\end{equation}

\noindent Iterating \eqref{alphaprimedef} and \eqref{betaprimedef} leads to a sequence of Bailey pairs, called the {\it Bailey chain}. Putting (\ref{alphaprimedef}) and (\ref{betaprimedef}) into (\ref{pairdef}) and letting $n \to \infty$ gives  

\begin{equation} \label{BL}
\sum_{n \geq 0} (b)_n(c)_n (aq/bc)^n \beta_n = \frac{(aq/b)_{\infty}(aq/c)_{\infty}}{(aq)_{\infty}(aq/bc)_{\infty}} \sum_{n \geq 0} \frac{(b)_n(c)_n(aq/bc)^n }{(aq/b)_n(aq/c)_n}\alpha_n.
\end{equation} 

For example, if we consider the Bailey pair relative to $q$ (see B(3) in \cite{Sl1})

\begin{equation} \label{alphab3}
\alpha_{n} = \frac{(1-q^{2n+1}) (-1)^{n} q^{\frac{3}{2}n^{2} + \frac{1}{2}n}}{1-q}
\end{equation} 
 
\noindent and
 
\begin{equation} \label{betab3}
\beta_{n} = \frac{1}{(q)_{n}},
\end{equation}

\noindent then one application of (\ref{alphaprimedef}) and (\ref{betaprimedef}) with $b$, $c \to \infty$ yields

\begin{equation} \label{newalphab3}
\alpha_{n}^{'} = \frac{(1-q^{2n+1}) (-1)^{n} q^{\frac{5}{2}n^{2} + \frac{3}{2}n}}{1-q}
\end{equation}

\noindent and

\begin{equation} \label{newbetab3}
\beta_{n}^{'} = \sum_{k=0}^{n} \frac{q^{k(k+1)}}{(q)_{k}(q)_{n-k}}
\end{equation}

\noindent while $l-2$ applications, $l>2$, of (\ref{alphaprimedef}) and (\ref{betaprimedef}) with $b$, $c \to \infty$ at each step produces

\begin{equation} \label{p2alphab3}
\alpha^{(l-2)}_{n} = \frac{(1-q^{2n+1}) (-1)^{n} q^{\frac{2l-1}{2}n^{2} + \frac{2l-3}{2}n}}{1-q}
\end{equation}

\noindent and

\begin{equation} \label{p2betab3}
\beta^{(l-2)}_{n} = \sum_{n=n_{l-1}, n_{l-2}, \dotsc, n_1 \geq 0} \frac{q^{\sum\limits_{k=1}^{l-2} n_k (n_k + 1)}}{(q)_{n_1} \prod\limits_{k=2}^{l-1} (q)_{n_k - n_{k-1}}}.
\end{equation}

\noindent Inserting (\ref{newalphab3}) and (\ref{newbetab3}) into (\ref{BL}), then letting $b \to \infty$ and $c=q$ gives

\begin{equation} \label{blb3}
\sum_{n, k \geq 0} (-1)^n \frac{q^{k(k+1) + \frac{n(n+1)}{2}} (q)_{n}}{(q)_{k} (q)_{n-k}} = \sum_{n \geq 0} q^{3n^2 + 2n} (1-q^{2n+1})
\end{equation} 

\noindent while substituting (\ref{p2alphab3}) and (\ref{p2betab3}) into (\ref{BL}), then letting $b \to \infty$ and $c=q$ leads to

\begin{equation} \label{genblb3}
\sum_{n_{l-1}, n_{l-2}, \dotsc, n_1 \geq 0} (-1)^{n_{l-1}}  \frac{q^{\sum\limits_{k=1}^{l-2} n_k (n_k + 1)+ \frac{n_{l-1} (n_{l-1}+1)}{2}} (q)_{n_{l-1}}} {(q)_{n_1} \prod\limits_{k=2}^{l-1} (q)_{n_k - n_{k-1}}} = \sum_{n \geq 0} q^{ln^2 + (l-1)n} (1 - q^{2n+1}).
\end{equation}

\section{$\Phi_{K}(q)$ and $\Phi_{-K}(q)$}

Let $K$ be an alternating knot with $c$ crossings and $D$ its associated diagram. We checkerboard $D$ with colors $A$ and $B$ such that the exterior $X$ is colored $A$ (here, we identify $D$ with the planar graph obtained by placing a vertex at each crossing and an edge at each arc) and let $\mathcal{T}_{K}$ be the Tait graph of $K$ (or, equivalently, of $D$). The reduced Tait graph $\mathcal{T}^{\prime}_{K}$ is obtained from $\mathcal{T}_{K}$ by replacing every set of two edges that connect the same two vertices by a single edge. Let $E(D)$ be the set of edges, $R$ the set of faces, $R_{A}$ the set of $A$-colored faces and $R_{B}$ the set of $B$-colored faces in $D$. The idea is to assign variables to each face of $D$, including $X$. Thus, we let

\begin{equation*}
S=\{s:R \to \mathbb{Z}: s(X)=0 \}.
\end{equation*}

For $F$, $F_{i}$ and $F_{j} \in R$, define $e(F)$ to be the number of edges of $F$, $cv(F_i, F_j)$ the number of common vertices and $ce(F_i, F_j)$ the number of common edges between $F_i$ and $F_j$. We now consider the functions $L: R \to \frac{1}{2} \mathbb{Z}$
and $Q: R \times R \to \mathbb{Z}$ given by

\begin{equation*}
  L(F) := \left\{
 \begin{array}{lll}
 1 & \text{if $F \in R_{B}$,}\\
    \frac{e(F)}{2} -1 & \text{if $F \in R_{A}$}
  \end{array} \right.
\end{equation*}

\noindent and

\begin{equation*}
Q(F_i, F_j) := \left\{
 \begin{array}{lll}
 0 & \text{if $i=j$, $F_i \in R_{B}$ or $i \neq j$, $F_i$, $F_j \in R_{A}$,}\\
  e(F_i)  & \text{if $i=j$, $F_i \in R_A$,} \\
   cv(F_i, F_j) & \text{if $i \neq j$, $F_i$, $F_j \in R_B$,} \\
    ce(F_i, F_j) & \text{if $i \neq j$, $F_i \in R_{B}$, $F_{j} \in R_A$ or $F_i \in R_{A}$, $F_j \in R_B$.} \\
  \end{array} \right.
\end{equation*}

We extend $s \in S$ to $E(D)$ by defining $s(e)$ to be the sum of the variables in adjacent faces. Furthermore, suppose $F \in R_B$ shares a common edge with the maximum number of faces in $R_A$. If $F$ is not unique, choose a face in $R_B$ that shares a common edge with the maximum number of faces in $R_A \setminus \{X\}$. If this latter face is not unique, choose from any of the remaining candidates of faces and let $F^{*}$ denote this choice. Finally, we let

\begin{equation*}
\Lambda := \{ s \in S: s(e) \geq 0, \forall e  \in E(D) \hspace{.1in} \text{and} \hspace{.1in} s(F^{*})=0 \}
\end{equation*}

\noindent and consider the functions $L^{\prime}: \Lambda \to \frac{1}{2} \mathbb{Z}^{|R| - 1}$ and $Q^{\prime}: \Lambda \to \frac{1}{2} \mathbb{Z}^{|R| - 1}$ defined by

\begin{equation*}
L^{\prime}(s) = \sum_{i=1}^{|R| - 1} L(F_i) s(F_i)
\end{equation*}

\noindent and

\begin{equation*}
Q^{\prime}(s) = \frac{1}{2} \sum_{1 \leq i, j \leq |R|-1} Q(F_i, F_j) s(F_i) s(F_j). 
\end{equation*}

The $q$-multisum $\Phi_{K}(q)$ is now given by (see Theorem 1.10 in \cite{gl})

\begin{equation*}
\Phi_{K}(q) = (q)_{\infty}^c S_{K} : = (q)_{\infty}^{c} \sum_{s \in \Lambda} (-1)^{2L^{\prime}(s)} \frac{q^{Q^{\prime}(s) + L^{\prime}(s)}}{\prod\limits_{e \in E(D)} (q)_{s(e)}}.
\end{equation*}

Let us illustrate this construction for $K=7_2$. We first consider

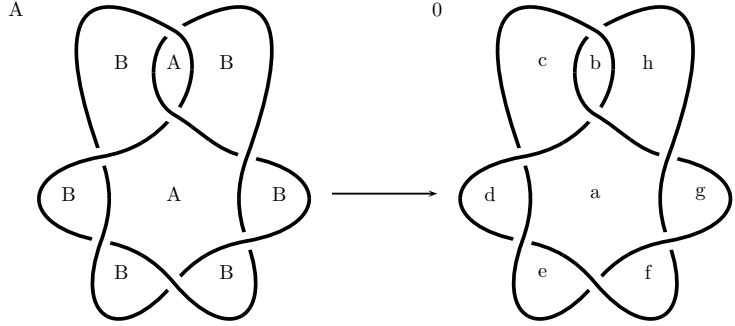
\begin{pspicture}(-2,2)(10,-3)
\rput(2,0){
\psscalebox{0.7}{	\psKnot[linewidth=1pt,knotscale=2](1,0){7-2}
	\rput(1,-1.5){A}
	\rput(1,1){A}
	\rput(0,1){B}
	\rput(2,1){B}
	\rput(-1,-1.5){B}
	\rput(3,-1.5){B}
	\rput(2,-3){B}
	\rput(0,-3){B}
	\rput(-2,2){A}
	\psline[linewidth=1pt]{->}(4,-1.5)(6,-1.5)
	\rput(9,-1.5){a}
	\rput(9,1){b}
	\rput(8,1){c}
	\rput(10,1){h}
	\rput(7,-1.5){d}
	\rput(11,-1.5){g}
	\rput(10,-3){f}
	\rput(8,-3){e}
	\rput(6,2){0}
	\psKnot[linewidth=1pt,knotscale=2](9,0){7-2}	}}
\end{pspicture}

In matrix notation, we have

\begin{equation*}
s = [c,d,e,f,g,h,a,b]^{T}, \quad L^{\prime}= [1,1,1,1,1,1,2,0],
\end{equation*}

\begin{equation}
Q^{\prime} = \begin{pmatrix}
0 & 1 & 0 & 0 & 0 & 2 & 1 & 1\\
1 & 0 & 1 & 0 & 0 & 0 & 1 & 0\\
0 & 1 & 0 & 1 & 0 & 0 & 1 & 0\\
0 & 0 & 1 & 0 & 1 & 0 & 1 & 0\\
0 & 0 & 0 & 1 & 0 & 1 & 1 & 0\\
2 & 0 & 0 & 0 & 1 & 0 & 1 & 1\\
1 & 1 & 1 & 1 & 1 & 1 & 6 & 0\\
1 & 0 & 0 & 0 & 0 & 1 & 0 & 2 \end{pmatrix}
\end{equation}

\noindent and

\begin{equation*}
\Lambda=\{[c,d,e,f,g,h,a,b]\in\mathbb{Z}^8 : a,b,c,d,e,f,g \geq 0, h=0\}.
\end{equation*}

\noindent Thus, in matrix notation, 

\begin{equation*}
\begin{aligned}
\Phi_{7_2}(q) & = (q)_{\infty}^7 S_{7_2} = (q)^{7}_{\infty}\sum_{s\in\Lambda}(-1)^{2L^{\prime} s}\frac{q^{s^{T}Q^{\prime} s + L^{\prime} s}}{\prod\limits_{e\in E(D)}(q)_{s(e)}}
\\
& = (q)^{7}_{\infty}\sum_{a,b,c,d,e,f,g \geq 0}\frac{q^{3a^{2} + 2a + b^2 + bc + ac + ad + ae + af + ag + cd + de + ef + fg + c + d + e + f + g}}{(q)_{a}(q)_{b}(q)_{c}(q)_{d}(q)_{e}(q)_{f}(q)_{g}(q)_{b+c}(q)_{a+c}(q)_{a+d}(q)_{a+e}(q)_{a+f}(q)_{a+g}}.
\end{aligned}
\end{equation*}

To compute $\Phi_{-K}(q)$, we repeat the above process but swap $A$ and $B$ faces while still imposing the condition that $s(X)=0$ and choosing $F^{*} \in R_{A}$. So, for $-K=-7_2$, 

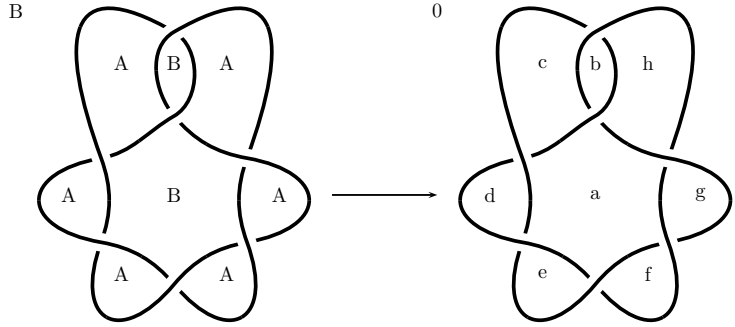
\begin{pspicture}(-2,2)(10,-3)
\rput(2,0){
\psscalebox{0.7}{	\rput(1,0){\psscalebox{-1 1}{\psKnot[linewidth=1pt,knotscale=2](0,0){7-2}}}
	\rput(1,-1.5){B}
	\rput(1,1){B}
	\rput(0,1){A}
	\rput(2,1){A}
	\rput(-1,-1.5){A}
	\rput(3,-1.5){A}
	\rput(2,-3){A}
	\rput(0,-3){A}
	\rput(-2,2){B}
	\psline[linewidth=1pt]{->}(4,-1.5)(6,-1.5)
	\rput(9,-1.5){a}
	\rput(9,1){b}
	\rput(8,1){c}
	\rput(10,1){h}
	\rput(7,-1.5){d}
	\rput(11,-1.5){g}
	\rput(10,-3){f}
	\rput(8,-3){e}
	\rput(6,2){0}
	\rput(9,0){\psscalebox{-1 1}{\psKnot[linewidth=1pt,knotscale=2](0,0){7-2}}}	}}

\end{pspicture}

Here, 

\begin{equation*}
s = [c,d,e,f,g,h,a,b]^{T}, \quad L^{\prime} = \biggl[\frac{1}{2},0,0,0,0,\frac{1}{2},1,1\biggr],
\end{equation*}

\begin{equation*}
Q^{\prime} = \begin{pmatrix}

3 & 0 & 0 & 0 & 0 & 0 & 1 & 1\\
0& 2 & 0 & 0 & 0 & 0 & 1 & 0\\
0 & 0 & 2 & 0 & 0 & 0 & 1 & 0\\
0 & 0 & 0& 2 & 0 & 0 & 1 & 0\\
0 & 0 & 0 & 0 & 2 & 0 & 1 & 0\\
0 & 0 & 0 & 0 & 0 & 2 & 1 & 1\\
1 & 1 & 1 & 1 & 1 & 1 & 0 & 1\\
1 & 0 & 0 & 0 & 0 & 1 & 1 & 0 \end{pmatrix}
\end{equation*}

\noindent and

\begin{equation*}
\Lambda=\{[a,b,c,d,e,f,g,h]\in\mathbb{Z}^8 : a,b,c,d,e,f,g \geq 0, h=0\}.
\end{equation*}

This gives us 

\begin{equation*}
\begin{aligned}
\Phi_{-7_2}(q) & = (q)_{\infty}^{7} S_{-7_2} = (q)^{7}_{\infty}\sum_{s \in \Lambda}(-1)^{2L^{\prime} s}\frac{q^{s^{T} Q^{\prime} s+ L^{\prime} s}}{\prod\limits_{e\in \epsilon}(q)_{s(e)}}
\\
=&(q)^{7}_{\infty}\sum_{a,b,c,d,e,f,g \geq 0}\frac{q^{a + b + ab+ ac + ad + ae + af + ag  + bc  + \frac{c(3c+1)}{2} + d^2 + e^2 + f^2 + g^2}}{(q)_{a}(q)_{b}(q)_{c}(q)_{d}(q)_{e}(q)_{f}(q)_{g}(q)_{a+c}(q)_{a+d}(q)_{a+e}(q)_{a+f}(q)_{a+g}(q)_{b+c}}.
\end{aligned}
\end{equation*}

Finally, by Theorem 2 in \cite{ad} or Corollary 1.12 in \cite{gl}, if the reduced Tait graphs of two alternating knots $K$ and $K'$ are isomorphic, then $\Phi_{K}(q) = \Phi_{K'}(q)$. Thus, in order to deduce Theorem \ref{main}, it suffices to verify the conjectural identities in the following cases: $5_1$, $5_2$, $6_2$, $7_1$, $7_2$, $7_4$, $7_7$, $8_2$, $8_4$, $K_{p}$, $p>0$, $T(2,p)$, $-3_1$, $-7_7$ and $-8_4$. For each of these 14 knots, we provide the checkerboard coloring, assignment of variables and (reduced) Tait graph.  \\

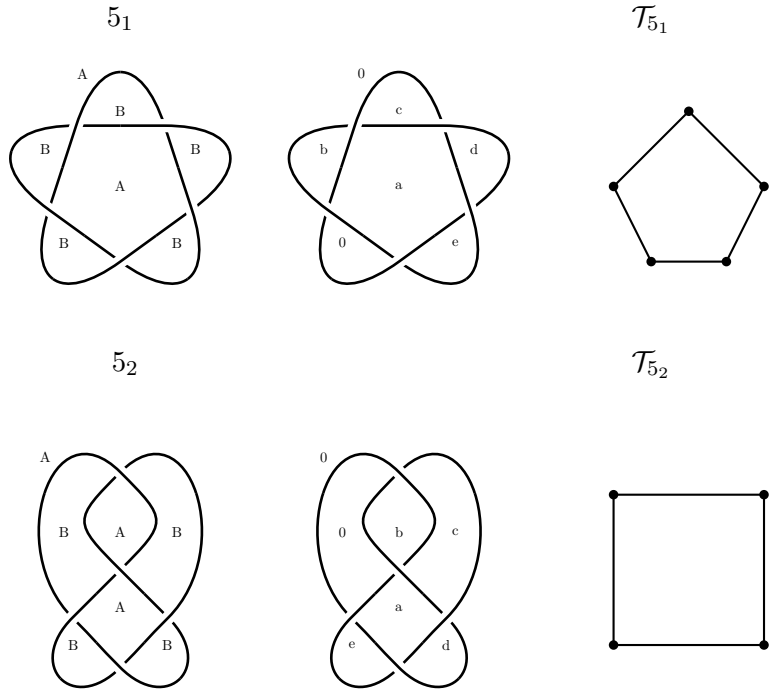
\begin{table}[ht] 
\centering 
\begin{tabular}{c c c c}  
$5_1$& & & $\mathcal{T}_{5_1}$\\
\begin{pspicture}(-1,2)(2,-2)
\psscalebox{0.5}{	\psKnot[linewidth=1pt,knotscale=2](1,0){5-1}
	\rput(1,0){A}
	\rput(-1,1){B}
	\rput(3,1){B}
	\rput(-0.5,-1.5){B}
	\rput(2.5,-1.5){B}
	\rput(1,2){B}
	\rput(0,3){A}}
\end{pspicture}
& & 
\begin{pspicture}(-1,2)(2,-2)
\psscalebox{0.5}{	\rput(1,0){a}
	\rput(-1,1){b}
	\rput(3,1){d}
	\rput(-0.5,-1.5){0}
	\rput(2.5,-1.5){e}
	\rput(1,2){c}
	\rput(0,3){0}
	\psKnot[linewidth=1pt,knotscale=2](1,0){5-1}	}
\end{pspicture} 

&
\begin{pspicture}(-1,2)(2,-2)
\pspolygon[showpoints=true](1,1)(0,0)(0.5,-1)(1.5,-1)(2,0)
\end{pspicture}

\\ \vspace{.2in}

$5_2$& & & $\mathcal{T}_{5_2}$\\
\begin{pspicture}(-1,2)(2,-2)
\psscalebox{0.5}{	\psKnot[linewidth=1pt,knotscale=2](1,0){5-2}
	\rput(1,-1){A}
	\rput(-0.5,1){B}
	\rput(2.5,1){B}
	\rput(-0.25,-2){B}
	\rput(2.25,-2){B}
	\rput(1,1){A}
	\rput(-1,3){A}}
\end{pspicture}
& & 
\begin{pspicture}(-1,2)(2,-2)
\psscalebox{0.5}{	
	\rput(1,-1){a}
	\rput(-0.5,1){0}
	\rput(2.5,1){c}
	\rput(-0.25,-2){e}
	\rput(2.25,-2){d}
	\rput(1,1){b}
	\rput(-1,3){0}
	\psKnot[linewidth=1pt,knotscale=2](1,0){5-2}	}
\end{pspicture}

&
\begin{pspicture}(-1,2)(2,-2)
\pspolygon[showpoints=true](0,1)(0,-1)(2,-1)(2,1)
\end{pspicture}
\end{tabular} 
\end{table}


\begin{table}
 \centering 
\begin{tabular}{c c c c} 
$6_2$& & & $\mathcal{T}_{6_2}$\\
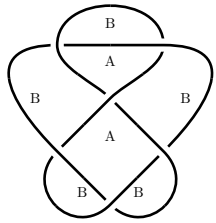
\begin{pspicture}(-1,2)(2,-2)
\psscalebox{0.5}{	
	\psKnot[linewidth=1pt, knotscale=2](1,0){6-2}
	\rput(-2,3){A}
	\rput(1,-1){A}
	\rput(1,1){A}
	\rput(-1,0){B}
	\rput(3,0){B}
	\rput(0.25,-2.5){B}
	\rput(1.75,-2.5){B}
	\rput(1,2){B}

}
\end{pspicture}
& &
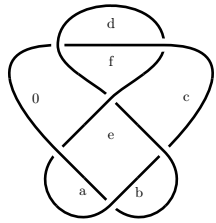
\begin{pspicture}(-1,2)(2,-2)
\psscalebox{0.5}{	
	\psKnot[linewidth=1pt, knotscale=2](1,0){6-2}
	\rput(-2,3){0}
	\rput(1,-1){e}
	\rput(1,1){f}
	\rput(-1,0){0}
	\rput(3,0){c}
	\rput(0.25,-2.5){a}
	\rput(1.75,-2.5){b}
	\rput(1,2){d}
}
\end{pspicture}

&
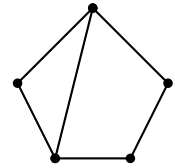
\begin{pspicture}(-1,2)(2,-2)
\pspolygon[showpoints=true](1,1)(0,0)(0.5,-1)(1.5,-1)(2,0)
\psline[showpoints=true](1,1)(0.5,-1)
\end{pspicture}

\\ \vspace{.2in}

$7_1$& & & $\mathcal{T}_{7_1}$\\
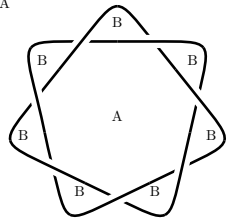
\begin{pspicture}(-1,2)(2,-2)
\psscalebox{0.5}{		\psKnot[linewidth=1pt, knotscale=2](1,0){7-1}
	\rput(-2,3){A}
	\rput(1,0){A}
	\rput(1,2.5){B}
	\rput(-1,1.5){B}
	\rput(3,1.5){B}
	\rput(-1.5,-0.5){B}
	\rput(3.5,-0.5){B}
	\rput(0,-2){B}
	\rput(2,-2){B}
}
\end{pspicture}
& &
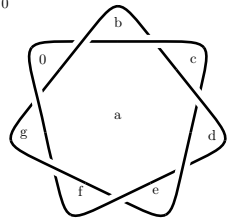
\begin{pspicture}(-1,2)(2,-2)
\psscalebox{0.5}{	
	\psKnot[linewidth=1pt, knotscale=2](1,0){7-1}
	\rput(-2,3){0}
	\rput(1,0){a}
	\rput(1,2.5){b}
	\rput(-1,1.5){0}
	\rput(3,1.5){c}
	\rput(-1.5,-0.5){g}
	\rput(3.5,-0.5){d}
	\rput(0,-2){f}
	\rput(2,-2){e}
}
\end{pspicture}

&
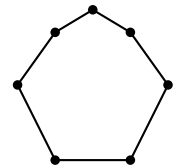
\begin{pspicture}(-1,2)(2,-2)
\pspolygon[showpoints=true](1,1)(0.5,0.7)(0,0)(0.5,-1)(1.5,-1)(2,0)(1.5,0.7)
\end{pspicture}
%

\\ \vspace{.2in}

$7_2$& & & $\mathcal{T}_{7_2}$\\
\begin{pspicture}(-1,2)(2,-2)
\psscalebox{0.5}{
	\psKnot[linewidth=1pt,knotscale=2](1,0){7-2}
	\rput(1,-1.5){A}
	\rput(1,1){A}
	\rput(0,1){B}
	\rput(2,1){B}
	\rput(-1,-1.5){B}
	\rput(3,-1.5){B}
	\rput(2,-3){B}
	\rput(0,-3){B}
	\rput(0,3){A}
}
\end{pspicture}
& & 
\begin{pspicture}(-1,2)(2,-2)
\psscalebox{0.5}{	
	\psKnot[linewidth=1pt,knotscale=2](1,0){7-2}
	\rput(1,-1.5){a}
	\rput(1,1){b}
	\rput(0,1){c}
	\rput(2,1){0}
	\rput(-1,-1.5){d}
	\rput(3,-1.5){g}
	\rput(2,-3){f}
	\rput(0,-3){e}
	\rput(0,3){0}
}
\end{pspicture}

&
\begin{pspicture}(-1,2)(2,-2)
\pspolygon[showpoints=true](0.5,1)(0,0)(0.5,-1)(1.5,-1)(2,0)(1.5,1)
\end{pspicture}

 
 \\ \vspace{.2in}

$7_4$& & & $\mathcal{T}_{7_4}$\\
\begin{pspicture}(-1,2)(2,-2)
\psscalebox{0.5}{	\psKnot[linewidth=1pt,knotscale=2](1,0){7-4}
	\rput(-2,3){A}
	\rput(2,0){A}
	\rput(0,0){A}
	\rput(1,1){B}
	\rput(1,-1){B}
	\rput(3,1){B}
	\rput(-1,1){B}
	\rput(-1,-1){B}
	\rput(3,-1){B}
}
\end{pspicture}
& & 
\begin{pspicture}(-1,2)(2,-2)
\psscalebox{0.5}{		\psKnot[linewidth=1pt,knotscale=2](1,0){7-4}
	\rput(-2,3){0}
	\rput(2,0){g}
	\rput(0,0){f}
	\rput(1,1){0}
	\rput(1,-1){c}
	\rput(3,1){a}
	\rput(-1,1){e}
	\rput(-1,-1){d}
	\rput(3,-1){b}
}
\end{pspicture}

&
\begin{pspicture}(-1,2)(2,-2)
\pspolygon[showpoints=true](0,1)(0,0)(0,-1)(2,-1)(2,0)(2,1)
\psline[showpoints=true](0,0)(2,0)
\end{pspicture}
\end{tabular}
\end{table}


 \begin{table}
 \centering 
\begin{tabular}{c c c c} 
$7_7$& & & $\mathcal{T}_{7_7}$\\
\begin{pspicture}(-1,2)(2,-2)
\psscalebox{0.5}{	\psKnot[linewidth=1pt,knotscale=2](1,0){7-7}
	\rput(-2,3){A}
	\rput(1,0.5){A}
	\rput(2.5,-0.5){A}
	\rput(-0.5,-0.5){A}
	\rput(3,1){B}
	\rput(-1,1){B}
	\rput(-1.5,-1){B}
	\rput(3.5,-1){B}
	\rput(1,-1){B}
}
\end{pspicture}
& &
\begin{pspicture}(-1,2)(2,-2)
\psscalebox{0.5}{		\psKnot[linewidth=1pt,knotscale=2](1,0){7-7}
	\rput(-2,3){0}
	\rput(1,0.5){e}
	\rput(2.5,-0.5){g}
	\rput(-0.5,-0.5){f}
	\rput(3,1){d}
	\rput(-1,1){a}
	\rput(-1.5,-1){b}
	\rput(3.5,-1){c}
	\rput(1,-1){0}
}
\end{pspicture}

&
\begin{pspicture}(-1,2)(2,-2)
\pspolygon[showpoints=true](1,1)(0,0)(0.5,-1)(1.5,-1)(2,0)
\psline[showpoints=true](0.5,-1)(2,0)
\psline[showpoints=true](0,0)(2,0)
\end{pspicture}

\\ \vspace{.1in}


$8_2$& & & $\mathcal{T}_{8_2}$\\
\begin{pspicture}(-1,2)(2,-2)
\psscalebox{0.5}{	\pscurve[linewidth=2pt] (0.9,-0.1)(0.2,-.9)(-1,0)(0,1)(0.1,.9)
\pscurve[linewidth=2pt] (0.4,.9)(0.5,0.9)(1,0)(1.8,-1.4)
\pscurve[linewidth=2pt] (1.9,-1.6)(3,-2.5)(3,-1.5)(2.9,-0.3)(3.1,-0.1)
\pscurve[linewidth=2pt] (3.3,0.1)(4,1.5)(3,1.5)(2.1,1.8)
\pscurve[linewidth=2pt] (1.9,1.9)(1,2)(0.5,1.5)(0.2,-.8)
\pscurve[linewidth=2pt] (0.2,-1)(1,-2)(2,-1.5)(2.9,-1.6)
\pscurve[linewidth=2pt] (3.1,-1.6)(4,-1.5)(3.2,0)(3,1.4)
\pscurve[linewidth=2pt] (3.1,1.6)(3.3,2)(2.5,2.5)(2,1.5)(1.1,0.1)

	\rput(0.5,0){A}
	\rput(2.5,0){A}
	\rput(-0.5,0){B}
	\rput(1.5,1.5){B}
	\rput(1.0,-1.5){B}
	\rput(2.75,2){B}
	\rput(2.75,-2){B}
	\rput(3.5,-1){B}
	\rput(3.5,1){B}
	\rput(-1,3){A}
}
\end{pspicture}
& &
\begin{pspicture}(-1,2)(2,-2)
\psscalebox{0.5}{		\pscurve[linewidth=2pt] (0.9,-0.1)(0.2,-.9)(-1,0)(0,1)(0.1,.9)
\pscurve[linewidth=2pt] (0.4,.9)(0.5,0.9)(1,0)(1.8,-1.4)
\pscurve[linewidth=2pt] (1.9,-1.6)(3,-2.5)(3,-1.5)(2.9,-0.3)(3.1,-0.1)
\pscurve[linewidth=2pt] (3.3,0.1)(4,1.5)(3,1.5)(2.1,1.8)
\pscurve[linewidth=2pt] (1.9,1.9)(1,2)(0.5,1.5)(0.2,-.8)
\pscurve[linewidth=2pt] (0.2,-1)(1,-2)(2,-1.5)(2.9,-1.6)
\pscurve[linewidth=2pt] (3.1,-1.6)(4,-1.5)(3.2,0)(3,1.4)
\pscurve[linewidth=2pt] (3.1,1.6)(3.3,2)(2.5,2.5)(2,1.5)(1.1,0.1)

	\rput(0.5,0){h}
	\rput(2.5,0){g}
	\rput(-0.5,0){a}
	\rput(1.5,1.5){b}
	\rput(1.0,-1.5){0}
	\rput(2.75,2){c}
	\rput(2.75,-2){f}
	\rput(3.5,-1){e}
	\rput(3.5,1){d}
	\rput(-1,3){0}
}
\end{pspicture}

&
\begin{pspicture}(-1,2)(2,-2)
\pspolygon[showpoints=true](1,1)(0.5,0.7)(0,0)(0.5,-1)(1.5,-1)(2,0)(1.5,0.7)
\psline[showpoints=true](1,1)(0,0)
\end{pspicture}


\\ \vspace{.2in}


$8_4$& & &  $\mathcal{T}_{8_4}$\\
\begin{pspicture}(-1,2)(2,-2)
\psscalebox{0.5}{	\pscurve[linewidth=2pt](1.5,0.8)(1,1)(0,0)(-0.5,-1)(0,-2)(1.5,-2.5)(3.1,-1.9)
\pscurve[linewidth=2pt](3.2,-1.7)(3,-1)(2.5,-1.3)(2,-1.5)(1.6,-1.4)
\pscurve[linewidth=2pt](1.4,-1.3)(1,-1)(0.1,-0.1)
\pscurve[linewidth=2pt](-0.1,0.1)(-0.5,1.5)(1.5,2.5)(3,2)(3.1,1.7)(3,1)(2.6,1.2)
\pscurve[linewidth=2pt](2.4,1.4)(2,1.5)(1.6,0.8)(1.5,0)(1.1,-0.9)
\pscurve[linewidth=2pt](1,-1.1)(1,-2)(1.5,-1.4)(2,-0.7)(2.5,-1.2)
\pscurve[linewidth=2pt](2.6,-1.4)(3.1,-1.8)(4,0)(3.4,1.6)(3.2,1.7)
\pscurve[linewidth=2pt](3,1.8)(2.8,1.8)(2.4,1.2)(2,0.7)(1.7,0.8)
	\rput(1,0){A}
	\rput(2,1){A}
	\rput(2.9,1.5){A}
	\rput(1.1,-1.5){A}
	\rput(2,-1){A}
	\rput(2.9,-1.35){A}
	\rput(3,0){B}
	\rput(1,1.5){B}
	\rput(0.5,-1.5){B}
	\rput(-1,3){A}
}
\end{pspicture}
& &
\begin{pspicture}(-1,2)(2,-2)
\psscalebox{0.5}{		\pscurve[linewidth=2pt](1.5,0.8)(1,1)(0,0)(-0.5,-1)(0,-2)(1.5,-2.5)(3.1,-1.9)
\pscurve[linewidth=2pt](3.2,-1.7)(3,-1)(2.5,-1.3)(2,-1.5)(1.6,-1.4)
\pscurve[linewidth=2pt](1.4,-1.3)(1,-1)(0.1,-0.1)
\pscurve[linewidth=2pt](-0.1,0.1)(-0.5,1.5)(1.5,2.5)(3,2)(3.1,1.7)(3,1)(2.6,1.2)
\pscurve[linewidth=2pt](2.4,1.4)(2,1.5)(1.6,0.8)(1.5,0)(1.1,-0.9)
\pscurve[linewidth=2pt](1,-1.1)(1,-2)(1.5,-1.4)(2,-0.7)(2.5,-1.2)
\pscurve[linewidth=2pt](2.6,-1.4)(3.1,-1.8)(4,0)(3.4,1.6)(3.2,1.7)
\pscurve[linewidth=2pt](3,1.8)(2.8,1.8)(2.4,1.2)(2,0.7)(1.7,0.8)

	\rput(1,0){e}
	\rput(2,1){d}
	\rput(2.9,1.5){c}
	\rput(1.1,-1.5){f}
	\rput(2,-1){g}
	\rput(2.9,-1.35){h}
	\rput(3,0){0}
	\rput(1,1.5){b}
	\rput(0.5,-1.5){a}
	\rput(-1,3){0}
}
\end{pspicture}

&
\begin{pspicture}(-1,2)(2,-2)
\pspolygon[showpoints=true](0,0)(2,0)(1,1)
\end{pspicture}

\\ \vspace{.2in}

$K_p, p>0$& & & $\mathcal{T}_{K_p}$, $p>0$\\
\begin{pspicture}(-1,2)(2,-2)
\psscalebox{0.5}{	\psline[linewidth=2pt](3,1)(3.5,1.4)
\pscurve[linewidth=2pt](3.6,1.6)(4,2.5)(2,2)(1.1,0.1)
\pscurve[linewidth=2pt](0.9,-0.1)(0.5,-0.5)(0,0)(-1,1.5)(0,2.5)(1.9,2)
\pscurve[linewidth=2pt](2.1,2)(3,2)(3.5,1.5)(4,1)

\pscurve[linewidth=2pt](3,-1)(3.5,-1.5)(4,-2.5)(2.1,-2.1)
\pscurve[linewidth=2pt](1.9,-1.9)(1,0)(0.5,0.5)(0.1,0.1)
\pscurve[linewidth=2pt](-0.1,-0.1)(-1,-1.5)(0,-2.5)(2,-2)(3,-2)(3.5,-1.6)
\psline[linewidth=2pt](3.6,-1.5)(4,-1)

	\pscircle[fillcolor=black,fillstyle=solid](3.5,0.5){0.1}
	\pscircle[fillcolor=black,fillstyle=solid](3.5,0){0.1}
	\pscircle[fillcolor=black,fillstyle=solid](3.5,-0.5){0.1}

	\rput(2,0){A}
	\rput(0.5,0){A}
	\rput(0,1.5){B}
	\rput(0,-1.5){B}
	\rput(3.5,2){B}
	\rput(3.5,-2){B}

	\rput(-1,3){A}
}
\end{pspicture}
& &
\begin{pspicture}(-1,2)(2,-2)
\psscalebox{0.5}{\psline[linewidth=2pt](3,1)(3.5,1.4)
\pscurve[linewidth=2pt](3.6,1.6)(4,2.5)(2,2)(1.1,0.1)
\pscurve[linewidth=2pt](0.9,-0.1)(0.5,-0.5)(0,0)(-1,1.5)(0,2.5)(1.9,2)
\pscurve[linewidth=2pt](2.1,2)(3,2)(3.5,1.5)(4,1)

\pscurve[linewidth=2pt](3,-1)(3.5,-1.5)(4,-2.5)(2.1,-2.1)
\pscurve[linewidth=2pt](1.9,-1.9)(1,0)(0.5,0.5)(0.1,0.1)
\pscurve[linewidth=2pt](-0.1,-0.1)(-1,-1.5)(0,-2.5)(2,-2)(3,-2)(3.5,-1.6)
\psline[linewidth=2pt](3.6,-1.5)(4,-1)

	\pscircle[fillcolor=black,fillstyle=solid](3.5,0.5){0.1}
	\pscircle[fillcolor=black,fillstyle=solid](3.5,0){0.1}
	\pscircle[fillcolor=black,fillstyle=solid](3.5,-0.5){0.1}

	\rput(2,0){\fontsize{12pt}{12pt}\selectfont$a$}
	\rput(0.5,0){\fontsize{12pt}{12pt}\selectfont$b$}
	\rput(0,1.5){\fontsize{12pt}{12pt}\selectfont$c_1$}
	\rput(0,-1.5){\fontsize{12pt}{12pt}\selectfont$0$}
	\rput(3.5,2){\fontsize{12pt}{12pt}\selectfont$c_2$}
	\rput(3.5,-2.2){\fontsize{12pt}{12pt}\selectfont$c_{2p-1}$}

	\rput(-1,3){A}	
}
\end{pspicture}

&
\begin{pspicture}(-1,2)(2,-2)
\psline[showpoints=true](2,-0.5)(1.5,-1)(0.5,-1)(0,-0.5)(0,0.5)(0.5,1)(1,1)
\psframe[linecolor=white, fillcolor=white,fillstyle=solid](1.1,0)(2,1)
\pscircle[fillcolor=black,fillstyle=solid](1.3,0.9){0.05}
\pscircle[fillcolor=black,fillstyle=solid](1.8,0.6){0.05}
\pscircle[fillcolor=black,fillstyle=solid](1.9,0.2){0.05}
\psscalebox{0.5}{\rput(2,0){$2p$ edges}}
\end{pspicture}
\end{tabular}  
\end{table}

\begin{table} 
\centering 
\begin{tabular}{c c c c} 
 $T(2,p), p>0$& & & $\mathcal{T}_{T(2,p)}$, $p>0$\\
\begin{pspicture}(-1,2)(2,-2)
\psscalebox{0.5}{	\rput(1,0){

\pscurve[linewidth=2pt](3,1.5)(2.8,2)(2.5,2.2)(1.8,2.2)(1,2.5)(0.55,2.9)
\pscurve[linewidth=2pt](0.45,3)(0,3.5)(-0.6,2.8)(-1,2.4)(-1.7,2.5)
\pscurve[linewidth=2pt](-1.9,2.5)(-2.7,2.4)(-2.5,1.4)(-2.4,0.8)(-3,0.1)
\pscurve[linewidth=2pt](-3.1,0)(-3.2,-0.8)(-2.7,-1.2)(-2,-1.5)(-1.8,-2.2)
\pscurve[linewidth=2pt](-1.8,-2.4)(-1.5,-3)(-1,-3)(-0.6,-2.8)(0,-2.5)(0.5,-2.9)
\pscurve[linewidth=2pt](0.6,-3.1)(1,-3.2)(1.5,-2.5)(1.6,-2.4)
\psline[linewidth=2pt](1.7,-2.2)(1.9,-1.7)

\psline[linewidth=2pt](2,1.5)(2,2)
\pscurve[linewidth=2pt](2,2.4)(1,3.3)(0.5,3)(0,2.5)(-0.6,2.5)
\pscurve[linewidth=2pt](-0.7,2.7)(-1,3.2)(-1.5,2.8)(-1.8,2.5)(-1.8,2)(-2.4,1.5)
\pscurve[linewidth=2pt](-2.6,1.5)(-3.3,1)(-3.1,0.1)(-2.5,-0.8)(-2.4,-1.1)
\pscurve[linewidth=2pt](-2.4,-1.3)(-2.5,-2.2)(-1.8,-2.3)(-0.6,-2.4)(-0.5,-2.6)
\pscurve[linewidth=2pt](-0.5,-2.8)(0,-3.5)(0.5,-3)(1,-2.2)(1.8,-2.3)(3,-2)(3,-1.5)

	\pscircle[fillcolor=black,fillstyle=solid](2.5,1){0.1}
	\pscircle[fillcolor=black,fillstyle=solid](2.5,0.5){0.1}
	\pscircle[fillcolor=black,fillstyle=solid](2.5,0){0.1}
	\pscircle[fillcolor=black,fillstyle=solid](2.5,-0.5){0.1}
	\pscircle[fillcolor=black,fillstyle=solid](2.5,-1){0.1}

	\rput(0,0){A}
	\rput(0,3){B}
	\rput(1.1,2.8){B}
	\rput(-1.1,2.8){B}
	\rput(-2.1,2.1){B}
	\rput(-2.75,0.9){B}
	\rput(0,-3){B}
	\rput(1.1,-2.7){B}
	\rput(-1.1,-2.7){B}
	\rput(-2.1,-2.0){B}
	\rput(-2.75,-0.8){B}
	\rput(-2,3){A}}
}
\end{pspicture}
& &
\begin{pspicture}(-1,2)(2,-2)
\psscalebox{0.5}{	\rput(1,0){\pscurve[linewidth=2pt](3,1.5)(2.8,2)(2.5,2.2)(1.8,2.2)(1,2.5)(0.55,2.9)
\pscurve[linewidth=2pt](0.45,3)(0,3.5)(-0.6,2.8)(-1,2.4)(-1.7,2.5)
\pscurve[linewidth=2pt](-1.9,2.5)(-2.7,2.4)(-2.5,1.4)(-2.4,0.8)(-3,0.1)
\pscurve[linewidth=2pt](-3.1,0)(-3.2,-0.8)(-2.7,-1.2)(-2,-1.5)(-1.8,-2.2)
\pscurve[linewidth=2pt](-1.8,-2.4)(-1.5,-3)(-1,-3)(-0.6,-2.8)(0,-2.5)(0.5,-2.9)
\pscurve[linewidth=2pt](0.6,-3.1)(1,-3.2)(1.5,-2.5)(1.6,-2.4)
\psline[linewidth=2pt](1.7,-2.2)(1.9,-1.7)

\psline[linewidth=2pt](2,1.5)(2,2)
\pscurve[linewidth=2pt](2,2.4)(1,3.3)(0.5,3)(0,2.5)(-0.6,2.5)
\pscurve[linewidth=2pt](-0.7,2.7)(-1,3.2)(-1.5,2.8)(-1.8,2.5)(-1.8,2)(-2.4,1.5)
\pscurve[linewidth=2pt](-2.6,1.5)(-3.3,1)(-3.1,0.1)(-2.5,-0.8)(-2.4,-1.1)
\pscurve[linewidth=2pt](-2.4,-1.3)(-2.5,-2.2)(-1.8,-2.3)(-0.6,-2.4)(-0.5,-2.6)
\pscurve[linewidth=2pt](-0.5,-2.8)(0,-3.5)(0.5,-3)(1,-2.2)(1.8,-2.3)(3,-2)(3,-1.5)
	\pscircle[fillcolor=black,fillstyle=solid](2.5,1){0.1}
	\pscircle[fillcolor=black,fillstyle=solid](2.5,0.5){0.1}
	\pscircle[fillcolor=black,fillstyle=solid](2.5,0){0.1}
	\pscircle[fillcolor=black,fillstyle=solid](2.5,-0.5){0.1}
	\pscircle[fillcolor=black,fillstyle=solid](2.5,-1){0.1}

	\rput(0,0){\fontsize{12pt}{12pt}\selectfont $a$}
	\rput(0,3){\fontsize{12pt}{12pt}\selectfont$0$}
	\rput(1.1,2.8){\fontsize{12pt}{12pt}\selectfont$b_{2p}$}
	\rput(-1.1,2.8){\fontsize{12pt}{12pt}\selectfont$b_1$}
	\rput(-2.1,2.1){\fontsize{12pt}{12pt}\selectfont$b_2$}
	\rput(-2.75,0.9){\fontsize{12pt}{12pt}\selectfont$b_3$}
	\rput(0,-3){\fontsize{12pt}{12pt}\selectfont$b_7$}
	\rput(1.05,-2.7){\fontsize{12pt}{12pt}\selectfont$b_8$}
	\rput(-1.1,-2.7){\fontsize{12pt}{12pt}\selectfont$b_6$}
	\rput(-2.1,-2.0){\fontsize{12pt}{12pt}\selectfont$b_5$}
	\rput(-2.75,-0.8){\fontsize{12pt}{12pt}\selectfont$b_4$}
	\rput(-2,3){\fontsize{12pt}{12pt}\selectfont$0$}}
}
\end{pspicture}

&
\begin{pspicture}(-1,2)(2,-2)
\psline[showpoints=true](2,-0.5)(1.5,-1)(0.5,-1)(0,-0.5)(0,0.5)(0.5,1)(1,1)
\psframe[linecolor=white,fillcolor=white,fillstyle=solid](1.1,0)(2,1)
\psframe[linecolor=white,fillcolor=white,fillstyle=solid](1.1,0)(2,1)
\pscircle[fillcolor=black,fillstyle=solid](1.3,0.9){0.05}
\pscircle[fillcolor=black,fillstyle=solid](1.8,0.6){0.05}
\pscircle[fillcolor=black,fillstyle=solid](1.9,0.2){0.05}
\psscalebox{0.5}{\rput(2,0){$2p+ 1$ edges}}
\end{pspicture}

\\ \vspace{.1in}

\end{tabular}  
\end{table}
\hfill
\begin{table} 
\centering 
\begin{tabular}{c c c c}

$-3_1$& & & $\mathcal{T}_{-3_1}$\\
\begin{pspicture}(-1,2)(2,-2)
\psscalebox{0.5}{	\psKnot[linewidth=1pt,knotscale=2](1,0){3-1}
	\rput(-2,3){B}
	\rput(1,0){B}
	\rput(1,2){A}
	\rput(-1,-0.5){A}
	\rput(3,-0.5){A}
}
\end{pspicture}
& &
\begin{pspicture}(-1,2)(2,-2)
\psscalebox{0.5}{		\psKnot[linewidth=1pt,knotscale=2](1,0){3-1}
	\rput(-2,3){0}
	\rput(1,0){a}
	\rput(1,2){0}
	\rput(-1,-0.5){b}
	\rput(3,-0.5){c}
}
\end{pspicture}

&
\begin{pspicture}(-1,2)(2,-2)
\psline[showpoints=true](1,1)(1,0)
\end{pspicture}

 
 \\ \vspace{.2in}

$-7_7$& & &$\mathcal{T}_{-7_7}$\\
\begin{pspicture}(-1,2)(2,-2)
\psscalebox{0.5}{	\psKnot[linewidth=1pt,knotscale=2](1,0){7-7}
	\rput(-2,3){B}
	\rput(1,0.5){B}
	\rput(2.5,-0.5){B}
	\rput(-0.5,-0.5){B}
	\rput(3,1){A}
	\rput(-1,1){A}
	\rput(-1.5,-1){A}
	\rput(3.5,-1){A}
	\rput(1,-1){A}
}
\end{pspicture}
& &
\begin{pspicture}(-1,2)(2,-2)
\psscalebox{0.5}{		\psKnot[linewidth=1pt,knotscale=2](1,0){7-7}
	\rput(-2,3){0}
	\rput(1,0.5){e}
	\rput(2.5,-0.5){g}
	\rput(-0.5,-0.5){f}
	\rput(3,1){d}
	\rput(-1,1){a}
	\rput(-1.5,-1){b}
	\rput(3.5,-1){c}
	\rput(1,-1){0}
}
\end{pspicture}

&
\begin{pspicture}(-1,2)(2,-2)
\pspolygon[showpoints=true](1,1)(0,0)(1,-1)(2,0)
\psline[showpoints=true](0,0)(2,0)
\end{pspicture}


\\ \vspace{.2in}


$-8_4$& & &  $\mathcal{T}_{-8_4}$\\
\begin{pspicture}(-1,2)(2,-2)
\psscalebox{0.5}{	\pscurve[linewidth=2pt](1.5,0.8)(1,1)(0,0)(-0.5,-1)(0,-2)(1.5,-2.5)(3.1,-1.9)
\pscurve[linewidth=2pt](3.2,-1.7)(3,-1)(2.5,-1.3)(2,-1.5)(1.6,-1.4)
\pscurve[linewidth=2pt](1.4,-1.3)(1,-1)(0.1,-0.1)
\pscurve[linewidth=2pt](-0.1,0.1)(-0.5,1.5)(1.5,2.5)(3,2)(3.1,1.7)(3,1)(2.6,1.2)
\pscurve[linewidth=2pt](2.4,1.4)(2,1.5)(1.6,0.8)(1.5,0)(1.1,-0.9)
\pscurve[linewidth=2pt](1,-1.1)(1,-2)(1.5,-1.4)(2,-0.7)(2.5,-1.2)
\pscurve[linewidth=2pt](2.6,-1.4)(3.1,-1.8)(4,0)(3.4,1.6)(3.2,1.7)
\pscurve[linewidth=2pt](3,1.8)(2.8,1.8)(2.4,1.2)(2,0.7)(1.7,0.8)
	\rput(1,0){B}
	\rput(2,1){B}
	\rput(2.9,1.5){B}
	\rput(1.15,-1.5){B}
	\rput(2,-1){B}
	\rput(2.9,-1.35){B}
	\rput(3,0){A}
	\rput(1,1.5){A}
	\rput(0.5,-1.5){A}
	\rput(-1,3){B}
}
\end{pspicture}
& &
\begin{pspicture}(-1,2)(2,-2)
\psscalebox{0.5}{		\pscurve[linewidth=2pt](1.5,0.8)(1,1)(0,0)(-0.5,-1)(0,-2)(1.5,-2.5)(3.1,-1.9)
\pscurve[linewidth=2pt](3.2,-1.7)(3,-1)(2.5,-1.3)(2,-1.5)(1.6,-1.4)
\pscurve[linewidth=2pt](1.4,-1.3)(1,-1)(0.1,-0.1)
\pscurve[linewidth=2pt](-0.1,0.1)(-0.5,1.5)(1.5,2.5)(3,2)(3.1,1.7)(3,1)(2.6,1.2)
\pscurve[linewidth=2pt](2.4,1.4)(2,1.5)(1.6,0.8)(1.5,0)(1.1,-0.9)
\pscurve[linewidth=2pt](1,-1.1)(1,-2)(1.5,-1.4)(2,-0.7)(2.5,-1.2)
\pscurve[linewidth=2pt](2.6,-1.4)(3.1,-1.8)(4,0)(3.4,1.6)(3.2,1.7)
\pscurve[linewidth=2pt](3,1.8)(2.8,1.8)(2.4,1.2)(2,0.7)(1.7,0.8)

	\rput(1,0){e}
	\rput(2,1){d}
	\rput(2.9,1.5){c}
	\rput(1.1,-1.5){f}
	\rput(2,-1){g}
	\rput(2.9,-1.35){h}
	\rput(3,0){0}
	\rput(1,1.5){b}
	\rput(0.5,-1.5){a}
	\rput(-1,3){0}
}
\end{pspicture}

&
\begin{pspicture}(-1,2)(2,-2)
\pspolygon[showpoints=true](1,1)(0.5,0.7)(0,0)(0.5,-1)(1.5,-1)(2,0)(1.5,0.7)
\psline[showpoints=true](1,1)(0.5,-1)
\end{pspicture}
\\
\end{tabular}  
\end{table}

\newpage

\section{Proof of Theorem \ref{main}}

We can now prove Theorem \ref{main}. 

\begin{proof}[Proof of Theorem \ref{main}] For $\Phi_{5_1}(q)$, it suffices to prove

\begin{equation} \label{51}
S_{5_1}:=\sum_{a,b,c,d,e \geq 0} (-1)^{a}\frac{q^{\frac{a(5a+3)}{2} + ab + ac + ad + ae + bc + cd + de + b + c + d + e}}{(q)_{a}(q)_{b}(q)_{c}(q)_{d}(q)_{e}(q)_{a+b}(q)_{a+c}(q)_{a+d}(q)_{a+e}} = \frac{1}{(q)^5_{\infty}} h_5.
\end{equation}

\noindent We now have

\begin{equation*}
\begin{aligned}
S_{5_1} &= \frac{1}{(q)_{\infty}}\sum_{i,j,k,b,c,d,e \geq 0} (-1)^{i+k}\frac{q^{\frac{3i(i+1)}{2} +  j^{2} + j + \frac{k(k+1)}{2} + 2ij + jk + ki + b + bc + c + ci + cd + d + di + dj + de + e + ei + ej + ek}}{(q)_{i}(q)_{j}(q)_{k}(q)_{b}(q)_{c}(q)_{d}(q)_{e}(q)_{i+b}(q)_{i+j+c}(q)_{i+j+k+d}} \\
& (\text{apply Lemma \ref{key} to the $a$-sum with $n=5$}) \\
& = \frac{1}{(q)_{\infty}^{2}}\sum_{i,j,k,b,c,d \geq 0}(-1)^{i+k}\frac{q^{\frac{3i(i+1)}{2} + j^{2} + j + \frac{k(k+1)}{2} + 2ij + jk + ki + b + bc + c + ci + cd + d + di + dj }}{(q)_{i}(q)_{j}(q)_{k}(q)_{b}(q)_{c}(q)_{d}(q)_{i+b}(q)_{i+j+c}} \\
& (\text{evaluate the $e$-sum with (\ref{e1})}) \\
& = \frac{1}{(q)_{\infty}^{5}}\sum_{i,j,k \geq 0} (-1)^{i+k}\frac{q^{\frac{3i(i+1)}{2} + j^{2} + j + \frac{k(k+1)}{2} + 2ij + jk + ki}}{(q)_{i}(q)_{j}(q)_{k}} \\
&  (\text{evaluate the $d$-sum, $c$-sum and $b$-sum with (\ref{e1})}) \\
& = \frac{1}{(q)_{\infty}^{5}}\sum_{i,j,k \geq 0} (-1)^{i+k}\frac{q^{\frac{i(i+1)}{2}+ j^{2} + j + \frac{k(k+1)}{2}+ jk}}{(q)_{i}(q)_{j-i}(q)_{k}} \quad (\text{shift $j \to j-i$}) \\
& = \frac{1}{(q)_{\infty}^{5}}\sum_{j,k \geq 0} (-1)^{k}\frac{q^{j^{2} + j + \frac{k(k+1)}{2}+ jk}}{(q)_{k}} \quad (\text{apply (\ref{qbt}) to the $i$-sum}) \\
& = \frac{1}{(q)_{\infty}^{4}}\sum_{j \geq 0}\frac{q^{j^{2} + j}}{(q)_{j}} \quad (\text{apply (\ref{e2}) to the $k$-sum}) \\
& = \frac{(q; q^5)_{\infty} (q^4; q^5)_{\infty} (q^5; q^5)_{\infty}}{(q)_{\infty}^5} \quad (\text{by (\ref{rr})}) \\
& =  \frac{1}{(q)^5_{\infty}} h_5 \quad (\text{apply (\ref{jtp}) with $q \to q^{5/2}$, $z=-q^{3/2}$}). 
\end{aligned}
\end{equation*}

For $\Phi_{5_2}(q)$, it suffices to prove

\begin{equation} \label{52}
S_{5_2}:= \sum_{a,b,c,d,e \geq 0} \frac{q^{2a^{2} + b^{2} + ac + ad + ae + bc + cd + de + a + c + d + e}}{(q)_{a}(q)_{b}(q)_{c}(q)_{d}(q)_{e}(q)_{b+c}(q)_{a+c}(q)_{a+d}(q)_{a+e}} = \frac{1}{(q)^5_{\infty}} h_4.
\end{equation}

\noindent Thus,

\begin{align}
S_{5_2} & = \frac{1}{(q)_{\infty}}\sum_{a,c,d,e \geq 0} \frac{q^{2a^{2} + ac + ad + ae + cd + de + a + c + d + e}}{(q)_{a}(q)_{c}(q)_{d}(q)_{e}(q)_{a+c}(q)_{a+d}(q)_{a+e}} \label{52step} \\
& (\text{evaluate the $b$-sum with (\ref{andy})}) \nonumber \\
& =\frac{1}{(q)_{\infty}^{2}}\sum_{i,j,c,d,e \geq 0} (-1)^{j} \frac{q^{i^{2} +i + \frac{j^{2} + j}{2} + ij + di + e(i+j) + cd + de + c + d + e}}{(q)_{i}(q)_{j}(q)_{c}(q)_{d}(q)_{e}(q)_{i+c}(q)_{i+j+d}}  \nonumber \\
& (\text{apply Lemma \ref{key} to the $a$-sum with $n=4$}) \nonumber \\
& =\frac{1}{(q)_{\infty}^{3}}\sum_{i,j,c,d\geq 0} (-1)^{j} \frac{q^{i^{2} +i + \frac{j^{2} + j}{2} + ij + di + cd + c + d}}{(q)_{i}(q)_{j}(q)_{c}(q)_{d}(q)_{i+c}} \quad (\text{evaluate the $e$-sum with (\ref{e1})})  \nonumber \\
& =\frac{1}{(q)_{\infty}^{5}}\sum_{i,j\geq 0} (-1)^{j} \frac{q^{i^{2} +i + \frac{j^{2} + j}{2} + ij}}{(q)_{i}(q)_{j}} \quad (\text{evaluate the $d$-sum and $c$-sum with (\ref{e1})})  \nonumber \\
& =\frac{1}{(q)_{\infty}^{5}}\sum_{i,j\geq 0} (-1)^{j} \frac{q^{i^{2} +i + \frac{j^{2} - j}{2} - ij}}{(q)_{i-j}(q)_{j}} \quad (\text{shift $i \to i-j$})  \nonumber \\
& = \frac{1}{(q)_{\infty}^{5}}\sum_{i\geq 0}(-1)^{i} q^{\frac{i^{2} +i}{2}} \quad (\text{apply (\ref{qbt}) to the $j$-sum, then use (\ref{negab})}) \nonumber \\
& =\frac{1}{(q)^5_{\infty}} h_4  \nonumber \\
& (\text{consider $i=2n$, $i=2n+1$, then let $n \to -n-1$ in the second resulting sum}).  \nonumber 
\end{align}

For $\Phi_{6_2}(q)$, it suffices to prove

\begin{equation*} \label{62}
S_{6_2}:= \sum_{a,b,c,d,e,f \geq 0}(-1)^{e} \frac{q^{2f^{2} + f + \frac{e(3e+1)}{2} + ab + af + bc + bf + cd + ce + cf + de + a + b + c + d}}{(q)_{a}(q)_{b}(q)_{c}(q)_{d}(q)_{e}(q)_{f}(q)_{a+f}(q)_{b+f}(q)_{c+e}(q)_{c+f}(q)_{d+e}} = \frac{1}{(q)^5_{\infty}} h_4. 
\end{equation*}

\noindent Thus,

\begin{equation*}
\begin{aligned}
S_{6_2} & =\frac{1}{(q)_{\infty}}\sum_{a,b,c,d,e,f \geq 0}(-1)^{e} \frac{q^{2f^{2} + f + \frac{e(e+1)}{2} + ab + af + bc + bf + cd + cf + de + a + b + c + d}}{(q)_{a}(q)_{b}(q)_{c}(q)_{d}(q)_{e}(q)_{f}(q)_{a+f}(q)_{b+f}(q)_{c+e}(q)_{c+f}} \\
& (\text{apply Lemma \ref{key} to the $e$-sum with $n=3$}) \\
& =\frac{1}{(q)^{2}_{\infty}}\sum_{a,b,c,e,f \geq 0}(-1)^{e} \frac{q^{2f^{2} + f + \frac{e(e+1)}{2} + ab + af + bc + bf + cf + a + b + c}}{(q)_{a}(q)_{b}(q)_{c}(q)_{e}(q)_{f}(q)_{a+f}(q)_{b+f}(q)_{c+f}} \\
& (\text{evaluate the $d$-sum with (\ref{e1})}) \\
\end{aligned}
\end{equation*}

\begin{equation*}
\begin{aligned}
& =\frac{1}{(q)_{\infty}}\sum_{a,b,c,f \geq 0}\frac{q^{2f^{2} + f + ab + af + bc + bf + cf + a + b + c}}{(q)_{a}(q)_{b}(q)_{c}(q)_{f}(q)_{a+f}(q)_{b+f}(q)_{c+f}} \quad (\text{evaluate the $e$-sum with (\ref{e2})}) \\
& = \frac{1}{(q)^5_{\infty}} h_4 \quad (\text{let $(a,b,c,f) \to (c,d,e,a)$, then proceed with (\ref{52step})}). 
\end{aligned} 
\end{equation*}

For $\Phi_{7_1}(q)$, it suffices to prove

\begin{equation*} \label{71}
\begin{aligned}
S_{7_1}&:=\sum_{a,b,c,d,e,f,g \geq 0} (-1)^{a}\frac{q^{\frac{a(7a+5)}{2} + ab + ac + ad + ae + af + ag + bc + cd + de + ef + fg + b + c + d + e+ f + g}}{(q)_{a}(q)_{b}(q)_{c}(q)_{d}(q)_{e}(q)_{f}(q)_{g}(q)_{a+b}(q)_{a+c}(q)_{a+d}(q)_{a+e}(q)_{a+f}(q)_{a+g}} \\
& = \frac{1}{(q)^7_{\infty}} h_7.
\end{aligned}
\end{equation*}

\noindent Thus, 

\begin{equation*}
\begin{aligned}
& S_{7_1} =\frac{1}{(q)_{\infty}} \sum_{i,j,k,l,m,b,c,d \geq 0} (-1)^{i+k+m} \frac{q^{ \frac{5i(i+1)}{2} + 2j(j+1) + \frac{3k(k+1)}{2} + l(l+1) + \frac{m(m+1)}{2}}}{(q)_{i}(q)_{j}(q)_{k}(q)_{l}(q)_{m}} \\
& \times \frac{q^{bc + cd + de + ef + fg + b + c + d + e+ f + g}}{(q)_{b} (q)_{c} (q)_{d} (q)_{e} (q)_{f} (q)_{g} } \\
& \times \frac{q^{4ij + 3ik + 2il + im + 3jk + 2jl + jm + 2kl + km + lm + ci + d(i+j) + e(i+j+k) + f(i+j+k+l) + g(i+j+k+l+m)}}{(q)_{b+i}(q)_{c+i+j}(q)_{d+i+j+k}(q)_{e+i+j+k+l}(q)_{f+i+j+k+l+m}} \\
& (\text{apply Lemma \ref{key} to the $a$-sum with $n=7$})  \\
& =\frac{1}{(q)_{\infty}^{7}} \sum_{i,j,k,l,m \geq 0} (-1)^{i+k+m} \frac{q^{ \frac{5i(i+1)}{2} + 2j(j+1) + \frac{3k(k+1)}{2} + l(l+1) + \frac{m(m+1)}{2}}}{(q)_{i}(q)_{j}(q)_{k}} \\ 
& \times \frac{q^{4ij + 3ik + 2il + im + 3jk + 2jl + jm + 2kl + km + lm}}{(q)_{l}(q)_{m}} \\
& (\text{evaluate the $g$-sum, $f$-sum, $e$-sum, $d$-sum, $c$-sum and $b$-sum with (\ref{e1})}) \\
& =\frac{1}{(q)_{\infty}^{7}} \sum_{i,j,k,l,m \geq 0} (-1)^{i+k+m} \frac{q^{ \frac{i(i+1)}{2} + 2j(j+1) + \frac{3k(k+1)}{2} + l(l+1) + \frac{m(m+1)}{2} + 3jk + 2jl + jm + 2kl + km + lm}}{(q)_{i}(q)_{j-i}(q)_{k}(q)_{l}(q)_{m}} \\
& (\text{shift $j \to j-i$}) \\
& =\frac{1}{(q)_{\infty}^{7}} \sum_{j,k,l,m \geq 0} (-1)^{k+m} \frac{q^{ 2j(j+1) + \frac{3k(k+1)}{2} + l(l+1) + \frac{m(m+1)}{2} + 3jk + 2jl + jm + 2kl + km + lm}}{(q)_{k}(q)_{l}(q)_{m}} \\
& (\text{evaluate the $i$-sum with (\ref{qbt})}) \\
\end{aligned}
\end{equation*}

\begin{equation*}
\begin{aligned}
& =\frac{1}{(q)_{\infty}^{7}} \sum_{j,k,l,m \geq 0} (-1)^{k+m} \frac{q^{ 2j(j+1) + \frac{k(k+1)}{2} + l(l+1) + \frac{m(m+1)}{2} + jk + 2jl + jm + lm}}{(q)_{k}(q)_{l-k}(q)_{m}} \quad (\text{shift $l \to l-k$}) \\
& =\frac{1}{(q)_{\infty}^{7}} \sum_{j,l,m \geq 0} (-1)^{m} \frac{q^{ 2j(j+1) + l(l+1) + \frac{m(m+1)}{2} + 2jl + jm + lm}(q^{1+j})_{l}}{(q)_{l}(q)_{m}}  \\ 
& (\text{evaluate the $k$-sum with (\ref{qbt})}) \\
&=\frac{1}{(q)_{\infty}^{6}} \sum_{j,l \geq 0} \frac{q^{ 2j(j+1) + l(l+1) + 2jl}}{(q)_{j}(q)_{l}} \quad (\text{evaluate the $m$-sum with (\ref{e2}) and simplfy}) \\
& = \frac{(q;q^{7})_{\infty}(q^{6};q^{7})_{\infty}(q^{7};q^{7})_{\infty}}{(q)_{\infty}^{7}} \quad (\text{by (\ref{ag}) with $k=3$, $n_1 = l$, $n_2 = j$}) \\
&  = \frac{1}{(q)^7_{\infty}} h_7 \quad (\text{by (\ref{jtp}) with $q \to q^{7/2}$, $z=-q^{5/2}$}).
\end{aligned} 
\end{equation*}

For $\Phi_{7_2}(q)$, it suffices to prove

\begin{equation} \label{72}
\begin{aligned}
S_{7_2} &:= \sum_{a,b,c,d,e,f,g \geq 0}\frac{q^{3a^{2} + 2a + b^2 + bc + ac + ad + ae + af + ag + cd + de + ef + fg + c + d + e+ f + g}}{(q)_{a}(q)_{b}(q)_{c}(q)_{d}(q)_{e}(q)_{f}(q)_{g}(q)_{b+c}(q)_{a+c}(q)_{a+d}(q)_{a+e}(q)_{a+f}(q)_{a+g}} \\
& = \frac{1}{(q)^7_{\infty}} h_6.
\end{aligned}
\end{equation}

\noindent Thus, 

\begin{equation*}
\begin{aligned}
S_{7_2} & =  \frac{1}{(q)_{\infty}^2} \sum_{i,j,k,l,c,d,e,f,g \geq 0} (-1)^{j+l} \frac{q^{2i(i+1) + \frac{3j(j+1)}{2} + k(k+1) + \frac{l(l+1)}{2}}}{(q)_{i}(q)_{j}(q)_{k}(q)_{l}} \\
& \times \frac{q^{3ij + 2ik + il + 2jk + jl + kl + di + e(i+j) + f(i+j+k) + g(i+j+k+l) + cd + de + ef + fg + c + d + e + f + g}}{(q)_{c+i}(q)_{d+i+j}(q)_{e+i+j+k}(q)_{f+i+j+k+l} (q)_c (q)_d (q)_e (q)_f (q)_g} \\
& (\text{evaluate the b-sum with (\ref{andy}) and apply Lemma \ref{key} to the $a$-sum with $n=6$}) \\
& =\frac{1}{(q)_{\infty}^{7}} \sum_{i,j,k,l \geq 0} (-1)^{j+l} \frac{q^{2i(i+1) + \frac{3j(j+1)}{2} + k(k+1) + \frac{l(l+1)}{2}+ 3ij + 2ik + il + 2jk + jl + kl }}{(q)_{i}(q)_{j}(q)_{k}(q)_{l}} \\
& (\text{evaluate the $g$-sum, $f$-sum, $e$-sum, $d$-sum and $c$-sum with (\ref{e1})}) \\
\end{aligned}
\end{equation*}

\begin{equation*}
\begin{aligned}
& =\frac{1}{(q)_{\infty}^{7}} \sum_{i,j,k,l \geq 0} (-1)^{j+l} \frac{q^{2i(i+1) + \frac{j(j-1)}{2} + k(k+1) + \frac{l(l+1)}{2} - ij + 2ik + il + kl }}{(q)_{i-j}(q)_{j}(q)_{k}(q)_{l}} \quad (\text{shift $i \to i-j$}) \\
& =\frac{1}{(q)_{\infty}^{7}} \sum_{i,k,l \geq 0} (-1)^{i+l} \frac{q^{\frac{3i(i+1)}{2} + k(k+1) + \frac{l(l+1)}{2} + 2ik + il + kl }}{(q)_{k}(q)_{l}} \\
& (\text{evaluate the $j$-sum with (\ref{qbt}), then use (\ref{negab})}) \\
& =\frac{1}{(q)_{\infty}^{7}} \sum_{i,k,l \geq 0} (-1)^{i+l} \frac{q^{\frac{3i(i+1)}{2} + k(k+1) + \frac{l(l-1)}{2} + 2ik - il - kl }}{(q)_{k-l}(q)_{l}} \quad (\text{shift $k \to k-l$}) \\
&=\frac{1}{(q)_{\infty}^{7}} \sum_{i,k \geq 0} (-1)^{i+k} \frac{q^{\frac{3i(i+1)}{2} + \frac{k(k+1)}{2} + ik}(q)_{k+i}}{(q)_{i}(q)_{k}} \\
& (\text{evaluate the $l$-sum with (\ref{qbt}), then use (\ref{negab}) and simplify}) \\
& =\frac{1}{(q)_{\infty}^{7}} \sum_{i,k \geq 0} (-1)^{k} \frac{q^{i(i+1) + \frac{k(k+1)}{2}} (q)_{k}}{(q)_{i}(q)_{k-i}} \quad (\text{shift $k \to k-i$}) \\
& = \frac{1}{(q)_{\infty}^{7}} \sum_{n \geq 0}q^{3n^{2} + 2n}(1-q^{2n+1}) \quad (\text{apply (\ref{blb3})}) \\
& =  \frac{1}{(q)^7_{\infty}} h_6 \quad (\text{let $n \to -n-1$ in the second sum}). 
\end{aligned} 
\end{equation*}

For $\Phi_{7_4}(q)$, it suffices to prove

\begin{equation*} \label{74} 
\begin{aligned}
S_{7_4}& := \sum_{a,b,c,d,e,f,g \geq 0} \frac{q^{2f^{2} + f +2g^{2} + g  + ab + ag + bc + bg + cd + cf + cg + de +df + ef + a + b + c + d + e}}{(q)_{a}(q)_{b}(q)_{c}(q)_{d}(q)_{e}(q)_{f}(q)_{g}(q)_{a+g}(q)_{b+g}(q)_{c+f}(q)_{c+g}(q)_{d+f}(q)_{e+f}} \\
& = \frac{1}{(q)^7_{\infty}} h_4^2. 
\end{aligned}
\end{equation*}

\noindent Thus, 

\begin{equation*}
\begin{aligned}
S_{7_4} & =\frac{1}{(q)^{2}_{\infty}}\sum_{a,b,c,d,e,i,j,k,l \geq 0}(-1)^{j+l} \frac{q^{i^2 + i + \frac{j(j+1)}{2} + k^2 + k + \frac{l(l+1)}{2} + ij + kl + di +e(i+j) + bk + c(k+l)  + ab + bc + cd + de + a + b + c + d + e}}{(q)_{a}(q)_{b}(q)_{c}(q)_{d}(q)_{e}(q)_{i}(q)_{j}(q)_{k}(q)_{l}(q)_{a+k}(q)_{b+k+l}(q)_{c+i}(q)_{d+i+j}} \\
& (\text{apply Lemma \ref{key} to the $f$-sum and $g$-sum with $n=4$}) \\
& =\frac{1}{(q)^{7}_{\infty}}\sum_{i,j,k,l \geq 0} (-1)^{j+l}\frac{q^{i^2 + i + \frac{j(j+1)}{2} + k^2 + k + \frac{l(l+1)}{2} + ij + kl}}{(q)_{i}(q)_{j}(q)_{k}(q)_{l}} \\
& (\text{evaluate the $e$-sum, $d$-sum, $c$-sum, $b$-sum and $a$-sum with (\ref{e1})}) \\
\end{aligned}
\end{equation*}

\begin{equation*}
\begin{aligned}
& =\frac{1}{(q)^{7}_{\infty}}\sum_{i,j,k,l \geq 0}(-1)^{j+l} \frac{q^{i^2 + i + \frac{j(j-1)}{2} + k^2 + k + \frac{l(l-1)}{2} - ij - kl}}{(q)_{i-j}(q)_{j}(q)_{k-l}(q)_{l}} \quad \text{(shift $i \to i-j$ and $k \to k-l$}) \\
& =\frac{1}{(q)^{7}_{\infty}}\sum_{i,k \geq 0}(-1)^{i+k}q^{\frac{i(i+1)}{2}+ \frac{k(k+1)}{2}} \quad (\text{evaluate the $j$-sum and $l$-sum with (\ref{qbt}), then use (\ref{negab})}) \\
& = \frac{1}{(q)^7_{\infty}} h_4^2 \quad (\text{as in the proof of (\ref{52})}). 
\end{aligned} 
\end{equation*}

For $\Phi_{7_7}(q)$, it suffices to prove

\begin{equation*} \label{77}
\begin{aligned}
S_{7_7} & := \sum_{a,b,c,d,e,f,g \geq 0} (-1)^{e+f+g}\frac{q^{\frac{3e^2}{2} + \frac{e}{2} +\frac{3f^2}{2} + \frac{f}{2} +\frac{3g^2}{2} + \frac{g}{2} + ab + ad + ae + af + bf + cd + cg + de + dg + a + b + c + d}}{(q)_{a}(q)_{b}(q)_{c}(q)_{d}(q)_{e}(q)_{f}(q)_{g}(q)_{a+e}(q)_{d+e}(q)_{a+f}(q)_{b+f}(q)_{c+g}(q)_{d+g}} \\
& =  \frac{1}{(q)_{\infty}^4}.
\end{aligned}
\end{equation*}

\noindent Thus, 

\begin{equation*}
\begin{aligned}
S_{7_7} & = \frac{1}{(q)_{\infty}^{3}} \sum_{a,b,c,d,e,f,g \geq 0} (-1)^{e+f+g}\frac{q^{\frac{e^2}{2} + \frac{e}{2} +\frac{f^2}{2} + \frac{f}{2} +\frac{g^2}{2} + \frac{g}{2} + ab + ad + ae + bf + cd + cg + a + b + c + d}}{(q)_{a}(q)_{b}(q)_{c}(q)_{d}(q)_{e}(q)_{f}(q)_{g}(q)_{d+e}(q)_{a+f}(q)_{d+g}} \\
& (\text{apply Lemma \ref{key} to $e$-sum, $f$-sum and $g$-sum with $n=3$}) \\
& = \frac{1}{(q)_{\infty}^{7}} \sum_{e,f,g \geq 0} (-1)^{e+f+g}\frac{q^{\frac{e(e+1)}{2}   +\frac{f(f+1)}{2} +\frac{g(g+1)}{2}}}{(q)_{e}(q)_{f}(q)_{g}} \\
& (\text{evaluate the $c$-sum, $b$-sum, $a$-sum and $d$-sum using (\ref{e1})}) \\
&  =  \frac{1}{(q)_{\infty}^4} \quad (\text{evaluate the $e$-sum, $f$-sum and $g$-sum using (\ref{e2})}).
\end{aligned} 
\end{equation*}

For $\Phi_{8_2}(q)$, it suffices to prove

\begin{equation*} \label{82}
\begin{aligned}
S_{8_2} & :=\sum_{a,b,c,d,e,f,g,h \geq 0}(-1)^{b}\frac{q^{3a^{2} + 2a + \frac{b(3b+1)}{2} + ad + ae + af + ag + ah + bc + bd + cd + de + ef + fg + gh + c + d + e+ f + g + h}}{(q)_{a}(q)_{b}(q)_{c}(q)_{d}(q)_{e}(q)_{f}(q)_{g}(q)_{h}(q)_{b+c}(q)_{b+d}(q)_{a+d}(q)_{a+e}(q)_{a+f}(q)_{a+g}(q)_{a+h}} \\
& = \frac{1}{(q)^7_{\infty}} h_6.
\end{aligned}
\end{equation*}

\noindent Thus, 

\begin{equation*}
\begin{aligned}
S_{8_2} & =\frac{1}{(q)_{\infty}}\sum_{a,b,c,d,e,f,g,h \geq 0}(-1)^{b}\frac{q^{3a^{2} + 2a + \frac{b(b+1)}{2} + ad + ae + af + ag + ah + bc + cd + de + ef + fg + gh + c + d + e+ f + g + h}}{(q)_{a}(q)_{b}(q)_{c}(q)_{d}(q)_{e}(q)_{f}(q)_{g}(q)_{h}(q)_{b+d}(q)_{a+d}(q)_{a+e}(q)_{a+f}(q)_{a+g}(q)_{a+h}} \\
& (\text{apply Lemma 2.5 to the $b$-sum with $n=3$}) \\
& =\frac{1}{(q)^{2}_{\infty}}\sum_{a,b,d,e,f,g,h \geq 0}(-1)^{b}\frac{q^{3a^{2} + 2a + \frac{b(b+1)}{2} + ad + ae + af + ag + ah + de + ef + fg + gh + d + e+ f + g + h}}{(q)_{a}(q)_{b}(q)_{d}(q)_{e}(q)_{f}(q)_{g}(q)_{h}(q)_{a+d}(q)_{a+e}(q)_{a+f}(q)_{a+g}(q)_{a+h}} \\
& (\text{evaluate the $c$-sum with (\ref{e1})}) \\
& =\frac{1}{(q)_{\infty}}\sum_{a,d,e,f,g,h \geq 0}\frac{q^{3a^{2} + 2a + ad + ae + af + ag + ah + de + ef + fg + gh + d + e+ f + g + h}}{(q)_{a}(q)_{d}(q)_{e}(q)_{f}(q)_{g}(q)_{h}(q)_{a+d}(q)_{a+e}(q)_{a+f}(q)_{a+g}(q)_{a+h}} \\
& (\text{evaluate the $b$-sum with (\ref{e2})}) \\
& = \frac{1}{(q)^7_{\infty}} h_6 \quad (\text{let $(a,d,e,f,g,h) \to (a,c,d,e,f,g)$, then follow the proof of (\ref{72})}).
\end{aligned} 
\end{equation*}

For $\Phi_{8_4}(q)$, it suffices to prove

\begin{equation*} \label{84}
\begin{aligned}
S_{8_4} &:= \sum_{a,b,c,d,e,f,g,h \geq 0} (-1)^{e}\frac{q^{\frac{e(3e+1)}{2} + ae + be + ab + a + b + c^{2} + bc + d^{2} +bd + f^{2} + af + g^{2} + ag + h^{2} + ah}}{(q)_{a}(q)_{b}(q)_{c}(q)_{d}(q)_{e}(q)_{f}(q)_{g}(q)_{h}(q)_{a+e}(q)_{a+f}(q)_{a+g}(q)_{a+h}(q)_{b+c}(q)_{b+d}(q)_{b+e}} \\
& = \frac{1}{(q)^7_{\infty}}.
\end{aligned}
\end{equation*}

\noindent Thus, 

\begin{equation*}
\begin{aligned}
S_{8_4} & = \frac{1}{(q)_{\infty}^{5}}\sum_{a,b,e \geq 0} (-1)^{e}\frac{q^{\frac{e(3e+1)}{2} + ae + be + ab + a + b}}{(q)_{a}(q)_{b}(q)_{e}(q)_{a+e}(q)_{b+e}} \\
& (\text{evaluate the $c$-sum, $d$-sum, $f$-sum, $g$-sum and $h$-sum with (\ref{andy})}) \\
& =\frac{1}{(q)_{\infty}^{6}}\sum_{a,b,e \geq 0} (-1)^{e}\frac{q^{\frac{e(e+1)}{2} + be + ab + a + b}}{(q)_{a}(q)_{b}(q)_{e}(q)_{a+e}} \quad (\text{apply Lemma \ref{key} to the $e$-sum with $n=3$}) \\
& =\frac{1}{(q)_{\infty}^{8}}\sum_{e \geq 0} (-1)^{e}\frac{q^{\frac{e(e+1)}{2}}}{(q)_{e}} \quad (\text{evaluate the $b$-sum and $a$-sum with (\ref{e1})}) \\
& = \frac{1}{(q)^7_{\infty}} \quad (\text{evaluate the $e$-sum with (\ref{e2})}).
\end{aligned} 
\end{equation*}

For $\Phi_{T(2,p)}(q)$ with $p>0$, it suffices to prove

\begin{equation*} \label{T2p}
\begin{aligned}
S_{T(2,p)} & := \sum_{a,b_{1},...,b_{2p} \geq 0} (-1)^{a}\frac{q^{\frac{a((2p+1)a+(2p-1))}{2} + a\sum\limits_{n=1}^{2p}b_{n} + \sum\limits_{n=1}^{2p-1}b_{n}b_{n+1} + \sum\limits_{n=1}^{2p}b_{n}}}{(q)_{a}\prod\limits_{n=1}^{2p}(q)_{b_{n}}(q)_{a+b_{n}}} \\
& = \frac{1}{(q)^{2p+1}_{\infty}} h_{2p+1}.
\end{aligned}
\end{equation*}

\noindent Thus, 

\begin{equation*}
\begin{aligned}
S_{T(2,p)} & = \frac{1}{(q)_{\infty}} \sum_{i_{1},...,i_{2p-1},b_{1},...,b_{2p} \geq 0} (-1)^{\sum\limits_{k=1}^{2p-1}\sum\limits_{j=1}^{k}i_{j}} \frac{q^{\frac{1}{2}\sum\limits_{k=1}^{2p-1} \bigl(\sum\limits_{j=1}^{k} i_{j} \bigr) \bigl( 1 + \sum\limits_{j=1}^{k} i_j \bigr)  + \sum\limits_{k=2}^{2p}\sum\limits_{j=1}^{k-1}b_{k}i_{j} + \sum\limits_{k=1}^{2p}b_{k} + \sum\limits_{k=1}^{2p-1}b_{k}b_{k+1}}}{\prod\limits_{k=1}^{2p-1}(q)_{i_{k}}\prod\limits_{k=1}^{2p-1}(q)_{b_{k}+\sum\limits_{j=1}^{k}i_{j}}\prod\limits_{k=1}^{2p}(q)_{b_{k}}}\\
& (\text{apply Lemma \ref{key} to the $a$-sum with $n=2p+1$}) \\
& =\frac{1}{(q)_{\infty}^{2p+1}} \sum_{i_{1},...,i_{2p-1} \geq 0} (-1)^{\sum\limits_{k=1}^{2p-1}\sum\limits_{j=1}^{k}i_{j}} \frac{q^{\frac{1}{2}\sum\limits_{k=1}^{2p-1} \bigl(\sum\limits_{j=1}^{k} i_{j} \bigr) \bigl( 1 + \sum\limits_{j=1}^{k} i_j \bigr)}}{\prod\limits_{k=1}^{2p-1}(q)_{i_{k}}} \\
& (\text{evaluate the $b_{2p}$-sum, $b_{2p-1}$-sum, $\dotsc$ and $b_1$-sum with (\ref{e1})}) \\
& =\frac{1}{(q)_{\infty}^{2p+1}} \sum_{i_{1},...,i_{2p-1} \geq 0} (-1)^{\sum\limits_{k=1}^{p}i_{2k-1}} \frac{q^{\frac{1}{2}\sum\limits_{k=1}^{p}i_{2k-1}(i_{2k-1}+1) + \sum\limits_{k=1}^{p}i_{2k-1}\sum\limits_{j=1}^{k-1}i_{2j} + \sum\limits_{k=1}^{p-1} \bigl( \sum\limits_{j=1}^{k}i_{2j} \bigr) \bigl(\sum\limits_{j=1}^{k}i_{2j}+1 \bigr)}}{ \prod\limits_{k=1}^{p} (q)_{i_{2k-1}} \prod\limits_{k=1}^{p-1}(q)_{i_{2k}  - i_{2k-1}}} \\
& (\text{shift $i_{2k} \to i_{2k} - i_{2k-1}$ for $k=1$, $2$, $\dotsc$, $p-1$}) \\
& =\frac{1}{(q)_{\infty}^{2p+1}} \sum_{i_{2},i_{4},...,i_{2p-2},i_{2p-1} \geq 0} (-1)^{i_{2p-1}} \frac{q^{\frac{i_{2p-1}(i_{2p-1}+1)}{2} + i_{2p-1}\sum\limits_{j=1}^{p-1}i_{2j} + \sum\limits_{k=1}^{p-1} \bigl(\sum\limits_{j=1}^{k}i_{2j} \bigr) \bigl(\sum\limits_{j=1}^{k}i_{2j}+1 \bigr)}}{(q)_{i_{2p-1}}} \\
& \times \prod\limits_{k=1}^{p-1} \frac{(q)_{\sum\limits_{j=1}^{k }i_{2j}}}{(q)_{\sum\limits_{j=1}^{k-1}i_{2j}}(q)_{i_{2k}}} \\
& (\text{evaluate the $i_1$-sum, $i_3$-sum, $\dotsc$ and $i_{2p-3}$-sum with (\ref{qbt}), then simplify}) \\
\end{aligned}
\end{equation*}

\begin{equation*}
\begin{aligned}
&  =\frac{1}{(q)_{\infty}^{2p+1}} \sum_{i_{2},i_{4},...,i_{2p-2},i_{2p-1} \geq 0} (-1)^{i_{2p-1}} q^{\frac{i_{2p-1}(i_{2p-1}+1)}{2} + i_{2p-1}\sum\limits_{j=1}^{p-1}i_{2j} + \sum\limits_{k=1}^{p-1} \bigl(\sum\limits_{j=1}^{k}i_{2j} \bigr) \bigl(\sum\limits_{j=1}^{k} i_{2j} + 1 \bigr)} \\
& \times \frac{(q)_{\sum\limits_{k=1}^{p-1} i_{2k}}}{(q)_{i_{2p-1}} \prod\limits_{k=1}^{p-1} (q)_{i_{2k}}} \quad (\text{simplify the product}) \\
& =\frac{1}{(q)_{\infty}^{2p}} \sum_{i_{2},i_{4},...,i_{2p-2} \geq 0}  \frac{q^{ \sum\limits_{k=1}^{p-1}(\sum\limits_{j=1}^{k}i_{2j})(\sum\limits_{j=1}^{k} i_{2j}+1)}}{\prod\limits_{k=1}^{p-1}(q)_{i_{2k}}} \quad (\text{evaluate the $i_{2p-1}$-sum with (\ref{e2})}) \\
& = \frac{1}{(q)^{2p+1}_{\infty}} h_{2p+1} \\
& (\text{let $n_j = i_{2j}$ and $k=p$ in (\ref{ag}) and $q \to q^{\frac{2p+1}{2}}$, $z=q^{\frac{2p-1}{2}}$ in (\ref{jtp})}). 
 \end{aligned} 
\end{equation*}

Before turning to the $\Phi_{K_p}(q)$, $p>0$ case, we note that for any given set of indices $\{i_1, i_2, \dotsc, i_n \}$, if we let $i_2 \to i_2 - i_1$, $i_3 \to i_3 - i_2$, $\dotsc$, $i_{n} \to i_n - i_{n-1}$, then 

\begin{equation} \label{sumtosum}
\sum_{k=1}^{n}\biggl(\sum_{j=1}^{k} i_{j} \biggr)\biggl(1+\sum_{j=1}^{k}i_{j}\biggr) - \frac{1}{2}\sum_{k=1}^{n}i_{k}(i_{k}+1) - \sum_{k=1}^{n}i_{k}\sum_{j=1}^{k-1}i_{j} = \sum_{k=1}^{n-1}i_{k}(i_{k}+1) + \frac{1}{2}i_{n}(i_{n}+1).
\end{equation}

For $\Phi_{K_p}(q)$ with $p>0$, it suffices to prove

\begin{equation*} \label{Kppos}
\begin{aligned}
S_{K_p}^{+} & :=\sum_{a,b,c_{1},...,c_{2p-1} \geq 0} \frac{q^{pa^{2} + (p-1)a + a\sum\limits_{n=1}^{2p-1}c_{n} + b^{2} + bc_{1} + \sum\limits_{n=1}^{2p-2}c_{n}c_{n+1} + \sum\limits_{n=1}^{2p-1}c_{n}}}{(q)_{a}(q)_{b}(q)_{b+c_{1}}\prod\limits_{n=1}^{2p-1}(q)_{c_{n}}(q)_{a+c_{n}}} \\
& = \frac{1}{(q)^{2p+1}_{\infty}} h_{2p}.
\end{aligned}
\end{equation*}

\noindent Thus, 

\begin{equation*}
\begin{aligned}
S_{K_p}^{+} & =\frac{1}{(q)_{\infty}}\sum_{a,c_{1},...,c_{2p-1} \geq 0} \frac{q^{pa^{2} + (p-1)a + a\sum\limits_{k=1}^{2p-1}c_{k} + \sum\limits_{k=1}^{2p-2}c_{k}c_{k+1} + \sum\limits_{k=1}^{2p-1}c_{k}}}{(q)_{a}\prod\limits_{k=1}^{2p-1}(q)_{c_{k}}(q)_{a+c_{k}}} \\
& (\text{evaluate the $b$-sum with (\ref{andy})}) \\
& = \frac{1}{(q)^{2}_{\infty}}\sum_{i_{1},...,i_{2p-2},c_{1},...,c_{2p-1} \geq 0} (-1)^{\sum\limits_{k=1}^{2p-2} \sum\limits_{j=1}^{k}i_{j}} \frac{q^{\frac{1}{2}\sum\limits_{k=1}^{2p-2} \bigl(\sum\limits_{j=1}^{k}i_{j} \bigr) \bigl(1+\sum\limits_{j=1}^{k}i_{j} \bigr) + \sum\limits_{k=2}^{2p-1}\sum\limits_{j=1}^{k-1}c_{k}i_{j} + \sum\limits_{k=1}^{2p-2}c_{k}c_{k+1} + \sum\limits_{k=1}^{2p-1}c_{k} }}{\prod\limits_{k=1}^{2p-2}(q)_{i_{k}}\prod\limits_{k=1}^{2p-2}(q)_{c_{k}+\sum\limits_{j=1}^{k}i_{j}}\prod\limits_{k=1}^{2p-1}(q)_{c_{k}}}\\
& (\text{apply Lemma \ref{key} to the $a$-sum with $n=2p$}) \\
& =\frac{1}{(q)^{2p+1}_{\infty}} \sum_{i_{1},...,i_{2p-2} \geq 0} (-1)^{\sum\limits_{k=1}^{2p-2}\sum\limits_{j=1}^{k}i_{j}} \frac{q^{\frac{1}{2}\sum\limits_{k=1}^{2p-2}\bigl(\sum\limits_{j=1}^{k}i_{j}\bigr)\bigl(1+\sum\limits_{j=1}^{k}i_{j}\bigr)}}{\prod\limits_{k=1}^{2p-2}(q)_{i_{k}}} \\
& (\text{evaluate the $c_{2p-1}$-sum, $c_{2p-2}$-sum, $\dotsc$ and $c_1$-sum with (\ref{e1})}) \\
& =\frac{1}{(q)^{2p+1}_{\infty}}\sum_{i_{1},...,i_{2p-2} \geq 0} (-1)^{\sum\limits_{k=1}^{p-1}i_{2k}} \frac{q^{\sum\limits_{k=1}^{p-1}\bigl(\sum\limits_{j=1}^{k}i_{2j-1}\bigr)\bigl(1+\sum\limits_{j=1}^{k}i_{2j-1}\bigr)  +  \frac{1}{2}\sum\limits_{k=1}^{p-1}i_{2k}(i_{2k} -1) - \sum\limits_{k=1}^{p-1}i_{2k}\sum\limits_{j=1}^{k}i_{2j-1}}}{\prod\limits_{k=1}^{p-1}(q)_{i_{2k-1}-i_{2k}}(q)_{i_{2k}}} \\
& (\text{shift $i_{2k-1} \to i_{2k-1} - i_{2k}$ for $k=1$, $2$, $\dotsc$, $p-1$}) \\
& =\frac{1}{(q)^{2p+1}_{\infty}}\sum_{i_{1},i_{3},...,i_{2p-3} \geq 0} (-1)^{\sum\limits_{k=1}^{p-1}i_{2k-1}} q^{\sum\limits_{k=1}^{p-1}\bigl(\sum\limits_{j=1}^{k} i_{2j-1}\bigr)\bigl(1+\sum\limits_{j=1}^{k}i_{2j-1}\bigr)  - \frac{1}{2}\sum\limits_{k=1}^{p-1}i_{2k-1}(i_{2k-1}+1) - \sum\limits_{k=1}^{p-1}i_{2k-1}\sum\limits_{j=1}^{k-1}i_{2j-1}} \\
& \times \frac{\prod\limits_{k=1}^{p-1}(q)_{\sum\limits_{j=1}^{k}i_{2j-1}}}{\prod\limits_{k=1}^{p-1}(q)_{i_{2k-1}}(q)_{\sum\limits_{j=1}^{k-1}i_{2j-1}}} \\
& (\text{evaluate the $i_2$-sum, $i_4$-sum, $\dotsc$ and $i_{2p-2}$-sum with (\ref{qbt}), then use (\ref{negab})}) \\
& = \frac{1}{(q)^{2p+1}_{\infty}}\sum_{i_{1},i_{3}...,i_{2p-3} \geq 0} (-1)^{\sum\limits_{k=1}^{p-1}i_{2k-1}} q^{\sum\limits_{k=1}^{p-1} \bigl(\sum\limits_{j=1}^{k} i_{2j-1}\bigr)\bigl(1+\sum\limits_{j=1}^{k}i_{2j-1}\bigr) - \frac{1}{2}\sum\limits_{k=1}^{p-1}i_{2k-1}(i_{2k-1}+1) - \sum\limits_{k=1}^{p-1}i_{2k-1}\sum\limits_{j=1}^{k-1}i_{2j-1}} \\
& \times \frac{(q)_{\sum\limits_{k=1}^{p-1}i_{2k-1}} }{\prod\limits_{k=1}^{p-1}(q)_{i_{2k-1}}} \quad (\text{simplify the product}) \\
\end{aligned}
\end{equation*}

\begin{equation*}
\begin{aligned}
& =\frac{1}{(q)^{2p+1}_{\infty}}\sum_{i_{1},i_{3}...,i_{2p-3} \geq 0} (-1)^{i_{2p-3}} \frac{q^{\sum\limits_{k=1}^{p-2}i_{2k-1}(1+i_{2k-1}) + \frac{1}{2}i_{2p-3}(i_{2p-3} + 1)}(q)_{i_{2p-3}}}{(q)_{i_{1}}\prod\limits_{k=2}^{p-1}(q)_{i_{2k-1}-i_{2k-3}}} \\
& (\text{let $i_3 \to i_3 - i_1$, $i_5 \to i_5 - i_3$, $\dotsc$, $i_{2p-3} \to i_{2p-3} - i_{2p-5}$, then apply (\ref{sumtosum})}) \\
&= \frac{1}{(q)^{2p+1}_{\infty}} \sum_{n \geq 0}q^{pn^{2} + (p-1)n}(1-q^{2n+1}) \quad (\text{apply (\ref{genblb3})}) \\
& = \frac{1}{(q)^{2p+1}_{\infty}} h_{2p} \quad (\text{let $n \to -n-1$ in the second sum}).
\end{aligned} 
\end{equation*}

For $\Phi_{-3_1}(q)$, it suffices to prove

\begin{equation*} \label{-31}
S_{-3_1} := \sum_{a,b,c \geq 0} \frac{q^{a + b^2 + c^2 + ab + ac}}{(q)_{a}(q)_{b}(q)_{c}(q)_{a+b}(q)_{a+c}} = \frac{1}{(q)^3_{\infty}}.
\end{equation*}

\noindent Thus, 

\begin{equation*}
\begin{aligned}
S_{-3_1} & = \frac{1}{(q)^{2}_{\infty}}\sum_{a \geq 0} \frac{q^{a }}{(q)_{a}} \quad (\text{evaluate the $b$-sum and $c$-sum with (\ref{andy})}) \\
& =\frac{1}{(q)^{3}_{\infty}} \quad (\text{evaluate the $a$-sum with (\ref{e1})}) .
\end{aligned}
\end{equation*}

For $\Phi_{-7_7}(q)$, it suffices to prove

\begin{equation*} \label{-77}
\begin{aligned}
S_{-7_7} & := \sum_{a,b,c,d,e,f,g \geq 0}(-1)^{e+f}\frac{q^{d^2 + \frac{e(3e+1)}{2} + \frac{f(3f+1)}{2} + g^2 + ab + ad  + ae + bc + be + bf + cf + cg+ a + b + c}}{(q)_{a}(q)_{b}(q)_{c}(q)_{d}(q)_{e}(q)_{f}(q)_{g}(q)_{a+d}(q)_{a+e}(q)_{b+e}(q)_{b+f}(q)_{c+f}(q)_{c+g}} \\
& = \frac{1}{(q)^{5}_{\infty}}.
\end{aligned}
\end{equation*}

Thus, 

\begin{equation*}
\begin{aligned} 
S_{-7_7} & = \frac{1}{(q)^{2}_{\infty}}\sum_{a,b,c,e,f \geq 0}(-1)^{e+f}\frac{q^{\frac{e(3e+1)}{2} + \frac{f(3f+1)}{2} + ab + ae + bc + be + bf + cf + a + b + c}}{(q)_{a}(q)_{b}(q)_{c}(q)_{e}(q)_{f}(q)_{a+e}(q)_{b+e}(q)_{b+f}(q)_{c+f}} \\
& (\text{evaluate the $d$-sum and $g$-sum with (\ref{andy})}) \\
& =\frac{1}{(q)^{4}_{\infty}}\sum_{a,b,c,e,f \geq 0} (-1)^{e+f}\frac{q^{\frac{e(e+1)}{2} + \frac{f(f+1)}{2}+ ab + bc + be + cf + a + b +c}}{(q)_{a}(q)_{b}(q)_{c}(q)_{e}(q)_{f}(q)_{a+e}(q)_{b+f}} \\
& (\text{apply Lemma \ref{key} to the $e$-sum and $f$-sum with $n=3$}) \\
& =\frac{1}{(q)^{5}_{\infty}}\sum_{a,b,e,f \geq 0} (-1)^{e+f}\frac{q^{\frac{e(e+1)}{2} + \frac{f(f+1)}{2}+ ab + be + a + b}}{(q)_{a}(q)_{b}(q)_{e}(q)_{f}(q)_{a+e}} \quad (\text{evaluate the $c$-sum with (\ref{e1})}) \\
& =\frac{1}{(q)^{7}_{\infty}}\sum_{e,f \geq 0} (-1)^{e+f}\frac{q^{\frac{e(e+1)}{2} + \frac{f(f+1)}{2}}}{(q)_{e}(q)_{f}} \quad (\text{evaluate the $b$-sum and $a$-sum with (\ref{e1})}) \\
& =\frac{1}{(q)^{5}_{\infty}} \quad (\text{evaluate the $e$-sum and $f$-sum with (\ref{e2})}).
\end{aligned}
\end{equation*}

For $\Phi_{-8_4}(q)$, it suffices to prove

\begin{equation*} \label{-84}
\begin{aligned}
S_{-8_4} & := \sum_{a,b,c,d,e,f,g,h \geq 0}(-1)^{g}\frac{q^{\frac{g(5g+3)}{2} + 2h^2 + ab + ah + bc + bh + cd + cg + ch + de + dg + ef + eg + fg + a + b + c + d + e + f + h}}{(q)_{a}(q)_{b}(q)_{c}(q)_{d}(q)_{e}(q)_{f}(q)_{g}(q)_{h}(q)_{a+h}(q)_{b+h}(q)_{c+g}(q)_{c+h}(q)_{d+g}(q)_{e+g}(q)_{f+g}} \\
& = \frac{1}{(q)^{8}_{\infty}} h_4 h_5. \end{aligned}
\end{equation*}

Thus, 

\begin{equation*} \label{-84}
\begin{aligned}
& S_{-8_4} \\
& =\frac{1}{(q)_{\infty}}\sum_{a,b,c,d,e,f,g,i,j \geq 0}(-1)^{g+j}\frac{q^{\frac{g(5g+3)}{2} + i(i+1) + \frac{j(j+1)}{2} + ij + ab + a(i+j) + bc + bi + cd + cg + de + dg + ef + eg + fg + a + b + c + d + e + f}}{(q)_{a}(q)_{b}(q)_{c}(q)_{d}(q)_{e}(q)_{f}(q)_{g}(q)_{i}(q)_{j}(q)_{b+i+j}(q)_{c+g}(q)_{c+i}(q)_{d+g}(q)_{e+g}(q)_{f+g}} \\
& (\text{apply Lemma \ref{key} to the $h$-sum with $n=4$}) \\
& =\frac{1}{(q)^{3}_{\infty}}\sum_{c,d,e,f,g,i,j \geq 0}(-1)^{g+j}\frac{q^{\frac{g(5g+3)}{2} + i(i+1) + \frac{j(j+1)}{2} + ij +  cd + cg + de + dg + ef + eg + fg + c + d + e + f}}{(q)_{c}(q)_{d}(q)_{e}(q)_{f}(q)_{g}(q)_{i}(q)_{j}(q)_{c+g}(q)_{d+g}(q)_{e+g}(q)_{f+g}} \\
& (\text{evaluate the $a$-sum and $b$-sum with (\ref{e1})}) \\
\end{aligned}
\end{equation*}

\begin{equation*}
\begin{aligned}
& =\frac{1}{(q)^{3}_{\infty}}\sum_{c,d,e,f,g,i,j \geq 0}(-1)^{g+j}\frac{q^{\frac{g(5g+3)}{2} + i(i+1) + \frac{j(j-1)}{2} - ij + cd + cg + de + dg + ef + eg + fg + c + d + e + f}}{(q)_{c}(q)_{d}(q)_{e}(q)_{f}(q)_{g}(q)_{i-j}(q)_{j}(q)_{c+g}(q)_{d+g}(q)_{e+g}(q)_{f+g}} \\
& (\text{shift $i \to i-j$}) \\
& =\frac{1}{(q)^{3}_{\infty}}\sum_{c,d,e,f,g,i \geq 0}(-1)^{g+i}\frac{q^{\frac{g(5g+3)}{2} + \frac{i(i+1)}{2} + cd + cg + de + dg + ef + eg + fg + c + d + e + f}}{(q)_{c}(q)_{d}(q)_{e}(q)_{f}(q)_{g}(q)_{c+g}(q)_{d+g}(q)_{e+g}(q)_{f+g}} \\
& (\text{evaluate the $j$-sum with (\ref{qbt}), then apply (\ref{negab})}) \\
& =\frac{1}{(q)^{4}_{\infty}}\sum_{c,d,e,f,i,r,s,t \geq 0}(-1)^{r+t+i}\frac{q^{\frac{3r(r+1)}{2} + s(s+1) + \frac{t(t+1)}{2} + 2rs + rt + st}}{(q)_{c}(q)_{d}(q)_{e}(q)_{f}(q)_{r}(q)_{s}(q)_{t}(q)_{c+r}(q)_{d+r+s}(q)_{e+r+s+t}} \\
& \times q^{\frac{i(i+1)}{2} + cd + de + dr + ef + e(r+s) + f(r+s+t) + c + d + e + f} \\
& (\text{apply Lemma \ref{key} to the $g$-sum with $n=5$}) \\
& =\frac{1}{(q)^{8}_{\infty}}\sum_{i,r,s,t \geq 0}(-1)^{r+t+i}\frac{q^{\frac{3r(r+1)}{2} + s(s+1) + \frac{t(t+1)}{2} + 2rs + rt + st + \frac{i(i+1)}{2}}}{(q)_{r}(q)_{s}(q)_{t}} \\
& (\text{evaluate the $f$-sum, $e$-sum, $d$-sum and $c$-sum with (\ref{e1})}) \\
& =\frac{1}{(q)^{8}_{\infty}}\sum_{i,r,s,t \geq 0}(-1)^{r+t+i}\frac{q^{\frac{r(r+1)}{2} + s(s+1) + \frac{t(t+1)}{2} + st + \frac{i(i+1)}{2}}}{(q)_{r}(q)_{s-r}(q)_{t}} \quad (\text{shift $s \to s-r$}) \\
& =\frac{1}{(q)^{8}_{\infty}}\sum_{i,s,t \geq 0}(-1)^{t+i}\frac{q^{s(s+1) + \frac{t(t+1)}{2} + st + \frac{i(i+1)}{2}}}{(q)_{t}} \quad (\text{evaluate the $r$-sum with (\ref{qbt})}) \\
& =\frac{1}{(q)^{7}_{\infty}}\sum_{i \geq 0}(-1)^{i}q^{\frac{i(i+1)}{2}}\sum_{s \geq 0}\frac{q^{s(s+1)}}{(q)_{s}} \\
& (\text{evaluate the $t$-sum with (\ref{e2}), then simplify}) \\
&  = \frac{1}{(q)^{8}_{\infty}} h_4 h_5 \\
& (\text{by (\ref{rr}), $q \to q^{5/2}$, $z=-q^{3/2}$ in (\ref{jtp}) and the proof of (\ref{52})}).
\end{aligned}
\end{equation*}

\end{proof}

\section*{Acknowledgements} The second author would like to thank Stavros Garoufalidis for his talk on June 23, 2011 at the Institut Math{\'e}matiques de Jussieu and the subsequent correspondence, the organizers (in particular, Ling Long and Holly Swisher) of the conference ``Applications of Automorphic Forms in Number Theory and Combinatorics", April 12--15, 2014 at LSU, Oliver Dasbach and Mustafa Hajij for their helpful comments and suggestions and George Andrews for his continued interest and encouragement. He would also like to thank the organizers (in particular, Adam Sikora) of the workshop ``Low-dimensional topology and number theory", August 17-23, 2014 at Oberwolfach for the opportunity to present this work. Finally, both authors are very grateful to the referee for their extremely careful reading of our paper.


\begin{thebibliography}{999}

\bibitem{andrews74}
G.E. Andrews, \emph{An analytic generalization of the Rogers-Ramanujan identities for odd moduli}, Proc. Nat. Acad. Sci. U.S.A \textbf{71} (1974), 4082--4085.

\bibitem{An2} 
G.E. Andrews, \emph{$q$-Series: Their Development and Application in Analysis, Number
Theory, Combinatorics, Physics, and Computer Algebra}, volume 66 of Regional
Conference Series in Mathematics. American Mathematical Society, Providence, RI,
1986.

\bibitem{An3}
G.E. Andrews, \emph{Bailey's transform, lemma, chains and tree}, Special functions 2000: current perspective and future directions (Tempe, AZ), 1--22, NATO Sci. Ser. II Math. Phys. Chem., {\bf 30}, Kluwer Acad. Publ., Dordrecht, 2001. 

\bibitem{and}
G.E. Andrews, \emph{Knots and $q$-series}, Ramanujan 125, 17--24, Contemp. Math., 627, Amer. Math. Soc., Providence, RI, 2014.

\bibitem{ab}
G.E. Andrews, D. Bowman, \emph{The Bailey transform and D. B. Sears}, Quaest. Math. \textbf{22} (1999), no. 1, 19--26.

\bibitem{ad}
C. Armond, O. Dasbach, \emph{Rogers-Ramanujan type identities and the head and tail of the colored Jones polynomial}, preprint available at \url{http://arxiv.org/abs/1106.3948}

\bibitem{bhl}
K. Bringmann, K. Hikami, J. Lovejoy, \emph{On the modularity of the unified WRT invariants of certain Seifert manifolds}, Adv. in Appl. Math. \textbf{46} (2011), no. 1-4, 86--93.

\bibitem{gar1}
S. Garoufalidis, \emph{Quantum knot invariants}, preprint available at \url{http://arxiv.org/abs/1201.3314}

\bibitem{gk}
S. Garoufalidis, R. Kashaev, \emph{From state integrals to $q$-series}, Math. Res. Lett., to appear.

\bibitem{gl}
S. Garoufalidis, T.  L{\^e}, \emph{Nahm sums, stability and the colored Jones polynomial}, Research in the Mathematical Sciences, to appear.

\bibitem{gt}
S. Garoufalidis, T. Vuong, \emph{Alternating knots, planar graphs and $q$-series}, Ramanujan J., to appear.

\bibitem{gr}
G. Gasper, M. Rahman, \emph{Basic hypergeometric series}, Encyclopedia of Mathematics and its Applications, \textbf{96}. Cambridge University Press, Cambridge, 2004.

\bibitem{mh}
M. Hajij, \emph{The tail of a quantum spin network}, preprint available at \url{http://arxiv.org/abs/1308.2369}

\bibitem{hik}
K. Hikami, \emph{Volume conjecture and asymptotic expansion of $q$-series}, Experiment. Math. \textbf{12} (2003), no. 3, 319--337.

\bibitem{hik1}
K. Hikami, \emph{Hecke type formula for unified Witten-Reshetikhin-Turaev invariants as higher-order mock theta functions}, Int. Math. Res. Not. IMRN 2007, no. 7, Art. ID rnm 022, 32pp.

\bibitem{lz}
R. Lawrence, D. Zagier, \emph{Modular forms and quantum invariants of $3$-manifolds}, Asian J. Math. \textbf{3} (1999), no. 1, 93--107.

\bibitem{Sl1}
L.J. Slater, \emph{A new proof of Rogers's transformations of infinite series}, Proc. London
Math. Soc. (2) {\bf 53} (1951) 460--475.

\bibitem{war}
S. O. Warnaar, \emph{50 years of Bailey's lemma}, Algebraic combinatorics and applications (G{\"o}{\ss}weinstein, 1999), 333--347, Springer, Berlin, 2001.

\end{thebibliography}
\end{document}